\documentclass[11pt]{amsart}
\usepackage{amsmath,mathtools}
\usepackage{amsmath, amsthm, amssymb, amsfonts, enumerate}
\usepackage[colorlinks=true,linkcolor=blue,urlcolor=blue]{hyperref}
\usepackage{dsfont}
\usepackage{color}
\usepackage{geometry}
\usepackage{todonotes}
\usepackage{epstopdf}
\usepackage{bbm}
\usepackage{geometry}
\usepackage[utf8]{inputenc}
\usepackage{bm}
\usepackage{amsfonts}
\usepackage{amsfonts}
\usepackage{textcomp}
\usepackage{amssymb}
\usepackage{float}
\usepackage{tikz}
\usepackage{epsfig}
\usepackage{amsmath}
\usepackage[english]{babel}
\usepackage{a4}
\usepackage{enumerate}
\usepackage{amsmath}
\newcommand{\stkout}[1]{\ifmmode\text{\sout{\ensuremath{#1}}}\else\sout{#1}\fi}

\usepackage{color}
\definecolor{red-}{rgb}{1.0,0.0,0.0}
\definecolor{grey}{rgb}{0.6, 0.6, 0.6}
\definecolor{brown}{rgb}{0.5,0.2,0.0}
\definecolor{brown-}{rgb}{0.0,0.1,1.0}
\definecolor{green-}{rgb}{0.0, 0.6, 0.0}
\definecolor{gold}{rgb}{0.8,0.7,0.0}
\definecolor{black}{rgb}{0.0,0.0,0.0}
\definecolor{DarkGreen}{rgb}{0.0,0.3,0.2}
\definecolor{LightGreen}{rgb}{0.8,1.0, 0.8}
\definecolor{yellow}{rgb}{0.9,0.9,0.0}
\definecolor{blue-}{rgb}{0.0,0.1,1.0}

\newtheorem{theorem}{Theorem}[section]
\newtheorem{remark}[theorem]{Remark}
\newtheorem{assumption}[theorem]{Assumption}
\newtheorem{lemma}[theorem]{Lemma}
\newtheorem{proposition}[theorem]{Proposition}
\newtheorem{corollary}[theorem]{Corollary}
\newtheorem{df}[theorem]{Definition}

\def \R{\mathbb{R}}

\def\d{\mathrm{d}}

\definecolor{red}{rgb}{1.0,0.0,0.0}

\definecolor{blu}{rgb}{0.0,0.0,1.0}

\definecolor{gre}{rgb}{0.03,0.50,0.03}

\title[Epidemic Control via Vaccination]{Optimal Vaccination in a  SIRS Epidemic Model}
\author[Federico]{Salvatore Federico}
\author[Ferrari]{Giorgio Ferrari}
\author[Torrente]{Maria-Laura Torrente}

\address{S.~Federico: Dipartimento di Economia, Universit\`a di Genova, Via F.\ Vivaldi 5, 16126, Genova, Italy}
\email{\href{mailto:salvatore.federico@unige.it}{salvatore.federico@unige.it}}
\address{G.~Ferrari: Center for Mathematical Economics (IMW), Bielefeld University, Universit\"atsstrasse 25, 33615, Bielefeld, Germany}
\email{\href{mailto:giorgio.ferrari@uni-bielefeld.de}{giorgio.ferrari@uni-bielefeld.de}}
\address{M.L.~Torrente: Dipartimento di Economia, Universit\`a di Genova, Via F.\ Vivaldi 5, 16126 Genova, Italy}
\email{\href{mailto:marialaura.torrente@economia.unige.it}{marialaura.torrente@economia.unige.it}}
\date{\today}

\numberwithin{equation}{section}

\begin{document}

\begin{abstract} 
We propose and solve an optimal vaccination problem within a deterministic compartmental model of SIRS type: the immunized population can become susceptible again, e.g.\ because of a not complete immunization power of the vaccine. A social planner thus aims at reducing the number of susceptible individuals via a vaccination campaign, while minimizing the social and economic costs related to the infectious disease. As a theoretical contribution, we provide a technical non-smooth verification theorem, guaranteeing that a semiconcave viscosity solution to the Hamilton-Jacobi-Bellman equation identifies with the minimal cost function, provided that the closed-loop equation admits a solution. Conditions under which the closed-loop equation is well-posed are then derived by borrowing results from the theory of \emph{Regular Lagrangian Flows}. From the applied point of view, we provide a numerical implementation of the model in a case study with quadratic instantaneous costs. Amongst other conclusions, we observe that in the long-run the optimal vaccination policy is able to keep the percentage of infected to zero, at least when the natural reproduction number and the reinfection rate are small. 
\end{abstract}

\maketitle

\smallskip

{\textbf{Keywords}}: SIRS model; optimal control; viscosity solution; non-smooth verification theorem; epidemic; optimal vaccination.

\smallskip

{\textbf{MSC2010 subject classification}}: 93C15, 49K15, 49L25, 92D30.
  
\smallskip

{\textbf{JEL classification}}: C61, I12, I18.


\section{Introduction}
\label{introduction}

During the recent Covid-19 pandemic, policymakers have been dealing in a first period with the implementation of severe lockdowns, while, in a second phase, with the massive vaccination policy of the susceptible population. Simultaneously, the scientific community experienced a renewed interest for epidemic mathematical models where an agent -- typically a social planner -- aims at taming the spread of a disease by designing lockdown policies and/or vaccination strategies that minimize social and economic costs. The starting point of the majority of this literature is the classical SIR model (cf.\ \cite{KermackMcKendrick}), where, at each point in time, each person in a population of $N$ individuals is either Susceptible, Infectious, or Recovered from a disease and can dynamically change her/his status according to a deterministic law of motions. 

Modeling a social planner's actions usually results into the introduction of control variables in the dynamics of the considered compartmental epidemic model. For example, the transmission rate becomes a control variable in the generalized SIR models considered in \cite{KruseStrack} and \cite{Micloetal}, and a controlled state-variable in the stochastic version of \cite{FedericoFerrari}; time-dependent lockdown policies directly affect the dynamics in the controlled SIRD (Susceptible-Infected-Recovered-Dead) model studied in \cite{Alvarez} and \cite{Calvia-etal}, as well as in the multi-group SIR model of \cite{Acemoglu}. With respect to the literature on optimal confinement policies, that on optimal vaccination has a longer history. The vaccination strategy that minimizes the social costs arising from the spread of a disease evolving according to the SIR model is studied in \cite{Hethcote-Waltman}. This contribution was communicated for publication on ``Mathematical Biosciences'' by Richard Bellman and it perhaps represents one of the earliest applications of the dynamic programming techniques in epidemiology. A related optimization problem is considered in \cite{Barrett-Hoel}, but in the context of a SIS (Susceptible-Infected-Susceptible) model. The comparison between compulsory vaccination and market allocation is studied in \cite{Brito-etal}. The seminal work \cite{Geoffard-Philipson} starts observing that ``of the roughly 40 vaccines on the market, only the smallpox vaccine has been successful in eradication'' by 1997. The study in \cite{Geoffard-Philipson} thus aims at understanding which forces prevent the eradication through vaccines of a disease, and the authors conclude that vaccinations yield a drop of infected individuals, which in turn leads to a drop in the demand for vaccines, which finally implies the return of the infectious disease. In the recent \cite{Glover-etal} it is studied how to best allocate a given time-varying supply of vaccines across individuals of different ages, and discuss the possible sub-optimality (in terms of  economic recovery) of the actual deployment path that prioritized older retired individuals to younger working people. A similar problem of optimal allocation of limited vaccines is considered in \cite{RaoBrandeau} as well.

In this paper we propose and study the problem of optimal vaccination against an infectious disease that evolves according to a generalized SIRS model. Differently to the classical SIR setting, where the compartment of recovered individuals is an absorbing state, in the SIRS model immunized people can  become susceptible again at a given rate $\eta>0$ (due, e.g., to a not complete immunization power of vaccines). We consider a social planner that aims at determining a vaccination strategy which reduces the number of susceptible individuals, while minimizing total social and economic costs. These are due, e.g., to the arrangement of vaccination hubs and to the employment of medical staff. 

It is well known that determining an explicit solution to control problems arising in epidemiology is extremely hard, if possible at all. As a matter of fact, those dynamic optimization problems are typically multidimensional and the dynamics of the controlled state-variables are nonlinear. Although we are not able to determine the expression of the optimal vaccination policy in closed form, in this paper we provide a thorough analysis of the optimal vaccination problem, which actually leads to the identification of easily verifiable sufficient conditions for the optimality of a candidate solution and to a numerical implementation in a relevant case study. This is accomplished as we explain in the following. First of all, we show that the assumed semiconcavity property of the instantaneous cost function is inherited by the problem's minimal cost function $V$, so that the latter is shown to be a semiconcave viscosity solution to the related Hamilton-Jacobi-Bellman (HJB) equation. A delicate technical analysis then allows to prove a verification theorem for non-smooth (i.e.\ viscosity) solutions to the HJB equation (see Theorem \ref{prop:convexanal} below). Even if the proof of this result is inspired by that of Theorem 3.9 in \cite{YongZhou1999}, the arguments therein needed to be thoroughly adapted and expanded to the present setting by properly exploiting the semiconcavity of $V$, and thus the properties of its supergradient (cf.\ Proposition \ref{prop:convexanal} below). The verification theorem assumes that a solution to the so-called closed-loop equation exists. Sufficient conditions for the well-posedness of the latter are then determined in Proposition \ref{prop:closedloop}, where results from the theory of Regular Lagrangian Flows (cf.\ \cite{Ambrosio1, Ambrosio2}) are employed. To the best of our knowledge, this is the first paper that combines the theory of Regular Lagrangian Flows with the study an optimal control problem. As a corollary of the verification theorem and of the existence of a solution to the closed-loop equation, we then obtain that the minimal cost function $V$ is indeed the unique semiconcave viscosity solution to the HJB equation.

The latter uniqueness results paves the way for a numerical study of the optimal vaccination problem. Indeed, as a complement to our theoretical analysis, we also provide a numerical implementation which is based on a recursion of the HJB equation, initialized by the null function. For the numerical exercise we assume a specification of the model with quadratic instantaneous costs, under which the conditions of the verification theorem are satisfied. Fixing the values of the model's parameters, we study the evolution of the optimal vaccination policy, the evolution of the instantaneous reproduction number, as well as the dynamics of the (optimally controlled) percentages of susceptible, infected, and immunized (recovered) individuals. A numerical result suggests that, following the optimal vaccination plan, the social planner is able in the long-run to keep the number of infected individuals equal to zero. In particular, this happens when the reinfection parameter or the natural reproduction number are sufficiently small. However, since the model prescribes reinfection at rate $\eta>0$, the disease cannot be eradicated in the strict sense, and the vaccination campaign cannot be terminated if the aim is to maintain zero infections. Hence, if the social planner wishes to stop vaccinating the population, other forms of control must be adapted, such as isolation. We also observe that the social planner is allowed to relax the vaccination policy only after a first time period where the maximal possible vaccination effort is made. As expected, the length of such an initial phase increases when the number of initial infected people increases.

The rest of this paper is organized as follows. Section \ref{sec:setting} presents the model and the problem's formulation. Section \ref{sec:mainresult} provides the theoretical analysis and the solution to the optimal vaccination problem, while Section \ref{sec:numerics} discusses the numerical results. Finally, conclusions are presented in Section \ref{sec:concl}.


\section{Problem Formulation}
\label{sec:setting}

\subsection{The Generalized SIRS Model}
\label{sec:SIR}

We model the spread of the infection by relying on a variation of the classical SIR model that dates back to the work by Kermack and McKendrick \cite{KermackMcKendrick}. The society has population $N$ and it consists of three different groups. The first group is formed by those people who are healthy, but susceptible to the disease; the second group contains those who are infected, while the last cohort consists of those who are immunized, that is recovered, dead or vaccinated. However, differently to the classical SIR model, we assume that, once immunized, an individual can become susceptible again, so that we face a SIRS epidemic model. We denote by $S(t)$ the percentage of individuals who are susceptible at time $t\geq0$, by $I(t)$ the percentage of infected, and by $R(t)$ the fraction of immunized (which, in the sequel, we will also call recovered). Clearly, $S(t) + I(t) + R(t)=1$ for all $t\geq0$.

The fraction of infected people grows at a rate which is proportional to the fraction of society that it is still susceptible to the disease. In particular, letting $\beta$ be the instantaneous transmission rate of the disease, during an infinitesimal interval of time $\d t$, each infected individual generates $\beta S(t)$ new infected individuals. It thus follows that the percentage of healthy individuals that get infected within $\d t$ units of time is $\beta I(t)S(t)$. During an infinitesimal interval of time $\d t$, the fraction of infected is reduced by $\gamma I(t)$, since infected recover from the disease at a rate $\gamma>0$. 

During an infinitesimal interval of time $\d t$, the fraction of susceptible people naturally decreases because, with rate $\beta$, susceptible become infected, as previously discussed. However, the decrease of the mass of susceptible people can be also controlled by a central planner via a vaccination policy that happens with rate $u(t)$, per unit of susceptible. On the other hand, the fraction of susceptible individuals can also grow, with a rate which is proportional to the fraction of society that it is immunized; namely, $\eta R(t)$ for some $\eta>0$. This can be due to the non perfect immunization effect of the vaccinations. 


Let $U:=[0,\overline{U}]$, for some $\overline{U} >0$, and assume that the vaccination rate $u(\cdot)$ belongs to the set
\begin{equation}
\label{eq:set-U}
\mathcal{U}:=\Big\{u:\, \mathbb{R}_+ \to U\,\,\text{measurable}\Big\}.
\end{equation} 
Then, according to the previous considerations, denoting by $f'$ the first-order derivative of a real-valued function $f$, the dynamics of $S(t), I(t)$ and $R(t)$ can be thus written as
\begin{equation}
\label{eq:S}
S'(t) = - \beta S(t) I(t) - u(t)S(t) + \eta R(t), \quad t>0, \qquad S(0)= s,
\end{equation}

\begin{equation}
\label{eq:I}
I'(t) = \beta S(t) I(t) - \gamma I(t), \quad t>0, \qquad I(0)= i,
\end{equation}
and
\begin{equation}
\label{eq:R}
R'(t) = \gamma I(t) - \eta R(t) + u(t)S(t), \quad t>0, \qquad R(0)= r,
\end{equation}
for given nonnegative $s, i$ and $r$ such that $s + i + r=1$. Given that $u(\cdot) \in \mathcal{U}$, the previous system of ODEs is well posed in the Carath\'eodory sense. Notice also that for any $t\geq0$, and for any choice of $u(\cdot) \in \mathcal{U}$, summing up the dynamics of $S, I$ and $R$ we have $(S' + I' + R')(t) = 0$ for all $t>0$, which then implies that $S(t) + I(t) + R(t) = 1$ for all $t\geq0$ as $s+i+r=1$.
Given this fact, it is then sufficient to consider only the dynamics of $(S,I)$, being $R(t)=1-S(t)-I(t)$. Hence, we obtain 
\begin{align}
& S'(t) = - \beta S(t) I(t) - u(t)S(t) + \eta (1-S(t)-I(t)), \quad t>0, \qquad S(0)= s, \label{eq:syst-dyn-S} 
\end{align}
and
\begin{align}
& I'(t) = \beta S(t) I(t) - \gamma I(t), \quad t>0, \qquad I(0)= i. \label{eq:syst-dyn-I} 
\end{align}
From \eqref{eq:syst-dyn-I}, one has for any $t\geq 0$
$$I(t) = i \exp\Big\{\int_0^t(\beta S(q) - \gamma) \d q\Big\},$$
so that $I(t)>0$ for any $t\geq0$ and for any $u(\cdot) \in \mathcal{U}$. Lemma \ref{lemm:SR-nonneg} in the Appendix shows that also $S(t)>0$ and $R(t)>0$ for any $t\geq0$, so that also $S(t)<1$, $R(t)<1$, $I(t)<1$ for any $t\geq0$. In the following, we assume that $(s,i) \in (0,1)^2$ are such that $s+i \in (0,1)$\footnote{The choice of considering $s+i<1$ -- i.e.\ of having an initial strictly positive percentage of immunized -- is only done in order to deal with an open set in the subsequent mathematical formulation of the problem. As a matter of fact, such a condition is not restrictive from the technical point of view as our results still apply if $s+i<\ell$, for some $\ell>1$, thus covering the case $s+i=1$ as well.}.


\subsection{The Social Planner Problem}
\label{sec:PB}

The epidemic generates social costs. These might arise because of lost gross domestic product (GDP) due to inability of working for infected, or because of an overwhelming of the national health-care system etc. Also, one can imagine that the more susceptible are, the larger is the probability of an additional wave of the epidemic and, therefore, of additional societal stress.
The social planner thus employs a vaccination policy aiming at reducing the number of susceptible individuals. These actions, however, come with a cost, which increases with the amplitude of the effort. The cost is due, e.g., to the arrangement of vaccination hubs and to the employment of medical staff.

The social planner thus aims at minimizing the cost functional
\begin{equation}
\label{eq:Probl0}
\int_0^{\infty} {e^{-r t}} C\big(S(t),I(t), u(t)\big) \d t.
\end{equation}
Here, $r\geq 0$ measures the social planner's time preferences, and $C:(0,1)^2\times \mathbb{R}_+ \to \mathbb{R}_+$ is a running cost function measuring the negative impact of the disease on the socio-economic state and on the  public health, as well as the costs arising because of vaccination policies. The following requirements are satisfied by $C$. Below, and in the rest of this paper, with a slight abuse of notation, we indicate by $|\,\cdot\,|$ both the absolute value and the Euclidean norm in $\mathbb{R}^2$.
\begin{assumption}
\label{ass:C}
\hspace{10cm}
\begin{itemize}
\item[(i)] $C$ is continuous and bounded. 
\item[(ii)] $u \mapsto C(s,i,u)$ is strictly convex for any $(s,i)\in(0,1)^2$.
\item[(iii)] $C$ is Lipschitz continuous with respect to $(s,i)$, uniformly in $u$; that is, there exists $K>0$ such that for any $u \in U$ we have that
$$|C(s,i,u) - C(s',i',u)| \leq K|(s,i)-(s',i')|, \quad \forall (s,i),(s',i')\in (0,1)^2.$$
\item[(iv)] $(s,i)\mapsto C(s,i,u)$ is  semiconcave
uniformly with respect to $u \in U$; that is, there exists ${K}>0$ such that for any $u \in U$ and any $\mu \in [0,1]$ one has $\forall (s,i),(s',i')\in (0,1)^2$
\begin{align*}
&\mu C(s,i,u) + (1-\mu) C(s',i',u) - C\big(\mu (s,i) + (1-\mu) (s,i'),u\big) \leq {K}\mu(1-\mu)|(s,i)-(s',i)'|^2.
\end{align*}
\end{itemize}
\end{assumption}
Without loss of generality, we also take $\inf_{\mathcal{M}\times U} C=0$. 
Convexity of $u\mapsto C(s,i,u)$ describes that marginal costs of vaccinations are increasing: E.g., an additional hiring of medical staff for vaccination might have a larger cost in a society that has already devoted many resources to the fight of the epidemics. Finally, the Lipschitz and semiconcavity property of $C(\cdot,\cdot,u)$ are technical requirements that will be important in the next section.

In order to tackle the social planner problem with techniques from dynamic programming, it is convenient to keep track of the initial values of $(S(\cdot),I(\cdot))$. We therefore let $x:=(s,i)$ and set 
\begin{equation}
\label{set:M}
\mathcal{M}:=\big\{x:=(s,i)\in \R^2:\,\, x\in (0,1)^2,\,\, s+i < 1\big\}.
\end{equation}
As we already observed, by Lemma  \ref{lemm:SR-nonneg}, the controlled dynamics of $(S(\cdot),I(\cdot))$ evolves within the the set $\mathcal{M}$. When needed, we stress the dependency of $(S(\cdot),I(\cdot))$ with respect to $x\in \mathcal{M}$ and $u(\cdot)\in\mathcal{U}$ by writing $(S^{x,u}(\cdot),I^{x,u}(\cdot))$. We shall also simply set $(S^{x}(\cdot),I^{x}(\cdot)):=(S^{x,0}(\cdot),I^{x,0}(\cdot))$ to denote the solutions to \eqref{eq:syst-dyn-S} and \eqref{eq:syst-dyn-I} when $u(\cdot) \equiv 0$.

Then, we introduce the problem's value function
\begin{equation}
\label{eq:V}
V(x) := \inf_{u(\cdot) \in \mathcal{U}}\int_0^{\infty} e^{-r t} C\big(S^{x,u}(t), I^{x,u}(t),u(t)\big) \d t, \ \ \ \ x\in \mathcal{M}.
\end{equation}
Since  $C$ is nonnegative and bounded, we have the following.
\begin{proposition}
 $V$ is well defined, nonnegative, and bounded. 
\end{proposition}
In the next section we will show that $V$ is semiconcave, solves the corresponding dynamic programming equation in the viscosity sense, and we will  also provide an optimal control in feedback form.


\section{The Solution to the Social Planner Problem}
\label{sec:mainresult}

Let $x=(s,i) \in \mathcal{M}$ and set
\begin{equation}
\label{eq:def-b}
b(x,u):= ( -\beta s i -u s+\eta (1-s-i), \beta s i-\gamma i), \quad x\in \mathcal{M}.
\end{equation}
In light of the dynamic programming principle (see, e.g., \cite{YongZhou1999}), we expect that $V$ should identify with a suitable solution to the Hamilton-Jacobi-Bellman (HJB) equation
\begin{equation}
\label{eq:HJB}
rv(x) = \mathcal{H} (x,Dv(x)), \quad x=(s,i) \in \mathcal{M},
\end{equation}
where  $Dv:=(v_s,v_i)$ denotes the gradient of $v$ (being $v_s$ and $v_i$ the partial derivatives in the $s$ and $i$ direction respectively) and 
$$\mathcal{H}(x,p)= \inf_{u \in U}\mathcal{H}_{{cv}}(x,p;u), \quad x=(s,i) \in \mathcal{M}, \ p=(p_s,p_i) \in \R^{2},$$
with
$$\mathcal{H}_{{cv}}(x,p;u)=\langle b(x,u), p\rangle+{C}(x,u), \quad x=(s,i) \in \mathcal{M}, \ p=(p_s,p_i)  \in \R^{2}.$$ 
Here, and in the sequel, $\langle \cdot, \cdot \rangle$ denotes the scalar product in $\mathbb{R}^2$.

Defining the linear operator 
\begin{equation}
\label{linearL}
(\mathcal{L}v)(x) :=\beta si (v_i(s,i)-v_s(s,i)+\eta (1-s-i)v_s(s,i)-\gamma i v_{i}(s,i), \ \ \ \ x=(s,i) \in \mathcal{M},
\end{equation}
we can separate the linear part in the HJB equation \eqref{eq:HJB} and write
\begin{equation}
\label{eq:HJBexplicit}
(r-\mathcal{L})v(x) = C^{\star}(x,v_{s}(x)),
\end{equation}
where 
$$C^{\star}(s,i,p_{s})=\inf_{u\in U} \left\{C(s,i,u)-usp_{s}\right\}, \ \ p_s\in\mathbb{R}.$$

For future frequent use, given an open set $\mathcal{O} \subseteq \mathbb{R}^2$ and a function $f:\mathcal{O} \to \mathbb{R}$, we denote by $D^+f$ (resp., $D^-f$) the supergradient (respectively, subgradient) of $f$; namely, for any $x \in \mathcal{O}$,
\begin{equation}
\label{eq:superdiff}
D^+f(x):=\Big\{ p \in \mathbb{R}^2:\, \liminf_{y \to x} \frac{f(y) - f(x) - \langle p, y-x \rangle}{|y-x|} \geq 0\Big\},
\end{equation}
and
\begin{equation}
\label{eq:subdiff}
D^-f(x):=\Big\{ p \in \mathbb{R}^2:\, \limsup_{y \to x} \frac{f(y) - f(x) - \langle p, y-x \rangle}{|y-x|} \leq 0\Big\}.
\end{equation}

Then, due to Assumption \ref{ass:C}, we have the following preliminary result which is the stationary version of  Theorems 7.4.11 and 7.4.14 in \cite{CS} (by taking $g\equiv 0$ therein).
\begin{theorem}
\label{thm:viscosity}
$V$ is  semiconcave on $\mathcal{M}$ and solves the HJB equation \eqref{eq:HJB} in the viscosity sense on $\mathcal{M}$; that is, 
$$
\mbox{(subsolution)} \ \ \ \ 
rV(x) \leq  \inf_{u \in U}\big\{\langle b(x,u), p \rangle +{C}(x,u)\big\}, \qquad \forall p:=(p_s,p_i)\in D^+V(x);$$
$$
\mbox{(supersolution)} \ \ \ \ 
rV(x)\geq  \inf_{u \in U}\big\{\langle b(x,u), p\rangle +{C}(x,u)\big\}, \qquad \forall p:=(p_s,p_i)\in D^-V(x).$$
\end{theorem}

Let now
\begin{equation}
\label{eq:def-Ustar}
U^{\star}(s,i,p_s):= \mbox{argmin}_{u \in U} \left\{C(s,i,u)- u s p_{s}\right\}, \quad (s,i,p_s) \in \mathcal{M} \times \mathbb{R},
\end{equation}
which, due to the strict convexity of $C(x,\cdot)$ (cf.\ Assumption \ref{ass:C}), exists unique and is continuous on $\mathcal{M} \times \mathbb{R}$ by the Berge's maximum theorem.
In light of the semiconcavity of $V$ and of the proof of the next verification theorem, let us now recall some well-known properties of semiconcave functions for the reader's convenience.
\begin{proposition}
\label{prop:convexanal}
Let $\mathcal{O}\subseteq \mathbb{R}^n$ and $f:\mathcal{O}\subseteq \mathbb{R}^n\to \mathbb{R}$ be semiconcave. Then the following hold true:
\begin{itemize}
\item[(a)] $D^{+}f(x)$ is nonempty, closed and convex for each $x\in\mathcal{O}$.
\item[(b)] The multi-valued map $x\mapsto D^{+}f(x)$ is locally bounded.
\item[(c)] $f$ is a.e.\ differentiable over $\mathcal{O}$.
\item[(d)] Given and fixed $x\in\mathcal{O}$, we have
$$
D^{+}f(x)= \mbox{co}\, D^{*}f(x),
$$
where `$\mbox{co}$' denotes the \emph{convex hull} and $D^{*}f(x)$ is the set of \emph{reachable gradients}:
$$D^{*}f(x):=\left\{\lim_{n\rightarrow \infty} Df(x_{n}): \ \exists Df(x_{n}) \ \mbox{and} \ x_{n}\rightarrow x\right\}. $$
\item[(e)] Given and fixed $\xi\in\mathbb{R}^{n}$ and setting
$$
f_{\xi}^{\pm}(x):=\lim_{r\rightarrow 0^{\pm}} \frac{ f(x+r\xi)- f(x)}{r}
$$
- which exists by semiconcavity - we have
\begin{equation}
\label{eq:dir-der}
f_{\xi}^{+}(x)=\min_{p\in D^{+}f(x)}\langle p,\xi\rangle, \ \ \  f_{\xi}^{-}(x)=\max_{p\in D^{+}f(x)}\langle p,\xi\rangle.
\end{equation}
\item[(f)] $D^{+}f(x)$ is compact and convex.
\end{itemize}
\end{proposition}
\begin{proof}
We simply provide precise reference for the reader's convenience. Item (a) follows by Proposition 3.1.5-(b) and Proposition 3.3.4-(c) in \cite{CS}; Item (b) is due to Theorem 2.1.7 in \cite{CS}; Item (c) is then implied by Item (b) and Rademacher's Theorem; Items (d) and (e) follow by Theorem 3.3.6 in \cite{CS}; Item (f) is due to Item (d) and the fact that the convex hull of a compact set - such as $D^{*}f(x)$, by definition - is compact (cf. Corollary A.1.7 in \cite{CS}).
\end{proof}

We are now ready to state and prove the main result of this section, namely a verification theorem for viscosity solutions to the HJB equation \eqref{eq:HJB}. Its proof is inspired by that of Theorem 3.9 in \cite{YongZhou1999}, which is, however, thoroughly adapted and expanded to the present setting by suitably exploiting the semiconcavity of $V$, and thus the subsequent properties of its supergradient (cf.\ Proposition \ref{prop:convexanal}).

In the following, given a semiconcave function $v:\mathcal{M}\to \R$ we set\footnote{Hereafter, the superscript $\pm$ means, as usual, either $+$ or $-$.}
$$\partial^{\pm}v_{s}:= v^{\pm}_{(1,0)}, \ \ \ \ 	 \ \ \ \ \ \partial^{*}_{s}v:=\{\partial^{+}_{s}v,\partial^{-}_{s}v\}.$$

\begin{theorem}
\label{thm:verification}[Non-smooth Verification Theorem]
Let $v:\mathcal{M} \to \mathbb{R}_+$ be  semiconcave, bounded  and nonnegative. Then the following hold true:
\begin{enumerate}
\item If $v$ is a viscosity subsolution to the HJB equation, then  $v\leq V$ on $\mathcal{M}$.
\item Recall \eqref{eq:def-Ustar}. Let $x:=(s,i)\in \mathcal{M}$, let  $v$ be a viscosity supersolution to the HJB equation, and let  $u^{\star}\in\mathcal{U}$ be such that, denoting by $(S^{\star}(\cdot),I^\star(\cdot))$ the state trajectory associated to $u^{\star}$, the following holds:

\begin{equation}
\label{eq:OC}
u^{\star}(t)\in U^{\star}\left(S^{\star}(t),I^{\star}(t),\partial^{*}_{s}v (S^{\star}(t),I^{\star}(t))\right), \ \ \ \mbox{for a.e.} \ t\geq 0.
\end{equation}
 Then $v(s,i)\geq J(s,i;u^{\star})$.
\end{enumerate}
\end{theorem}
\begin{proof}

Recall $b$ as in \eqref{eq:def-b}. Given an  admissible control $u(\cdot)$ and the corresponding controlled trajectory $X(\cdot)=(S(\cdot), I(\cdot))$, for future use throughout this proof, define the set
\begin{align*}
&\mathcal{T}= \mathcal{T}_{{u(\cdot)}}:=\big\{t>0: v(X(\cdot))\,\, \text{is differentiable and}\,\, \lim_{h \rightarrow 0+ }\frac{1}{h} \int_{t-h}^t b(X(s),u(s)) \d s \\
&\hspace{2cm} = \lim_{h \rightarrow 0+ }\frac{1}{h} \int_{t}^{t+h} b(X(s),u(s)) \d s = b(X(t),u(t))\big\}.
\end{align*}
Notice that $\mathcal{T}$ has full measure. Indeed, $b(X(\cdot),u(\cdot)) \in L^1_{\text{loc}}(\mathbb{R}^+)$; furthermore, because of the Lipschitz property of $t \mapsto X(t)$ and the locally Lipschitz property of $v$, which follows by semiconcavity, we have that $t\mapsto v(X(t))$ is locally Lipschitz continuous, and, hence, it is differentiable a.e. 
\vspace{0.25cm} 
 
\emph{Step 1.} Let $\bar u(\cdot)$ be an admissible control and $\bar X(\cdot):=(\bar S(\cdot),\bar I(\cdot))$ be the associated state trajectory. 
In order to simplify notation, from now on, we set $\bar b(t):=b(\bar{X}(t),\bar u(t))$. 

Let now $t\in \mathcal{T}$ and let $\bar p(t)\in D^{+}  v(\bar X(t))$. We then have:
\begin{align}
\label{der-v}
\frac{\d}{\d t} e^{-r t}v(\bar X(t)) &=  \lim_{h\rightarrow 0^{+}}  \frac{-e^{-r (t-h)}v(\bar X(t-h))+e^{-rt}v(\bar X (t))}{h}  \\ 
&=e^{-r t} \lim_{h\rightarrow 0^{+}}  \frac{e^{r h}v(\bar X (t)) -e^{r h}v(\bar X(t-h)) +v(\bar X (t)) - e^{r h}v(\bar X (t))}{h} \nonumber \\
&= e^{-r t} \bigg(\lim_{h\rightarrow 0^{+}}e^{rh}\frac{v(\bar X (t-h) + h\bar b(t) + o(h)) - v(\bar X(t-h))}{h} - rv(\bar X(t))\bigg) \nonumber \\
& \geq e^{-rt}\left(\langle \bar p(t), \bar b(t)\rangle - rv(\bar X(t)) \right), \nonumber
\end{align}
where the last inequality follows by the fact that $\bar p(t)\in D^{+}  v(\bar X(t))$.

Since now $v$ is a viscosity subsolution to the HJB equation, we find from \eqref{der-v}
\begin{equation}
\label{eq:der-v-2}
\frac{\d}{\d t} e^{-r t}v(\bar X(t)) \geq - rv(\bar X(t))+  \langle \bar p(t), \bar b(t)\rangle \geq - C(\bar{X}(t),\bar u(t)).
\end{equation}
On the other hand, because of the Lipschitz property of $v(\bar X(\cdot))$ we can write, for $\bar X(0)=x=(s,i) \in \mathcal{M}$,
\begin{equation}
\label{eq:int-by-parts}
e^{-rT}v(\bar X(T))-v(x)=\int_{0}^{T}\frac{\d}{\d t} e^{-r t}v(\bar X(t)) \d t,
\end{equation}
and picking a measurable selection $t\mapsto \bar{p}(t)\in D^{+} v(\bar{X}(t))$ (see, e.g., page 277 in \cite{YongZhou1999}), and using \eqref{eq:der-v-2} in \eqref{eq:int-by-parts}, we find
\begin{align*}
&v(x)\leq e^{-rT}v(\bar{X}(T))+ \int_{0}^{T}e^{{-rt}}C(\bar{X}(t),\bar u(t)) \d t.
\end{align*}
Since  $v$ and $C$ are bounded, we safely take the limit as $T\rightarrow\infty$ obtaining
$$v(x) \leq \int_{0}^{\infty}e^{{-rt}}C(\bar{X}(t),\bar u(t)) \d t.$$
By the arbitrariness of $\bar u$ and $x$ it follows that $v \leq V$ on $\mathcal{M}$, as claimed.
\vspace{0.25cm}

\emph{Step 2.} For $x \in \mathcal{M}$ given and fixed, let 
$$P^{+}(x):= \mbox{argmin}_{p\in D^{+}v(x)}\langle p, (1,0)\rangle, \ \ \ \ P^{-}(x):= \mbox{argmax}_{p\in D^{+}v(x)}\langle p, (1,0)\rangle $$
and notice that 
$$
\langle p, (1,0)\rangle =\partial^{+}_{s}v(x), \ \ \forall p \in P^+(x), \ \ \ \ \langle p, (1,0)\rangle =\partial^{-}_{s}v(x), \ \ \forall p \in P^-(x).
$$ 
Hence, since $P^{\pm}(x)$ is closed, bounded, and convex, it must have the structure
$$
P^{\pm}(x)= \{\partial^{\pm}_{s}v(x)\}\times [\underline{p}^{\pm}_{i}(x), \bar{p}^{\pm}_{i}(x)],
$$
for some $-\infty<\underline{p}^{\pm}_{i}(x)\leq  \bar{p}^{\pm}_{i}(x)<\infty$. The point $(\partial^{\pm}_{s}v(x), \underline{p}^{\pm}_{i}(x))$ is thus an extremal point for the convex compact set $D^{+}v(x)$ (cf.\ Proposition \ref{prop:convexanal}-(e)), and by Proposition \ref{prop:convexanal}-(d) there exists a sequence $x^{\pm}_{n}\rightarrow x$ such that
$$
\exists Dv(x^{\pm}_{n})\rightarrow  (\partial^{\pm}_{s}v(x), \underline{p}^{\pm}_{i}(x))\in D^{+}v(x).
$$
Since $v$ is a viscosity supersolution we have 
$$
rv(x^{\pm}_{n})\geq \inf_{u\in U } \big\{\langle b(x^{\pm}_n,u), Dv(x^{\pm}_{n})\rangle +C(x^{\pm}_{n},u)\big\},
$$ 
which, taking the limit as $n\rightarrow \infty$, yields
\begin{align}
\label{eq:limitHJBbis}
rv(x)&\geq  \inf_{u\in U}\{\langle  b(x,u), (\partial^{\pm}_{s}v(x), \underline{p}_{i}^{\pm}(x)))\rangle +C(x,u)\}.
\end{align}
Using the definition of $U^{\star}$ as in \eqref{eq:def-Ustar}, \eqref{eq:limitHJBbis} gives
\begin{equation}
\label{eq:limitHJB}
rv(x)\geq  \langle  b(x,U^{\star}(x,\partial^{\pm}_sv(x)), (\partial^{\pm}_{s}v(x), \underline{p}_{i}^{\pm}(x))\rangle +C(x,U^{\star}(x,\partial^{\pm}_sv(x))).
\end{equation}

Now, let $u^{\star}(\cdot)$  and $X^{\star}(\cdot):=(S^{\star}(\cdot),I^{\star}(\cdot))$ as in the claim, and take $t \in \mathcal{T}^{\star}=\mathcal{T}_{u^\star(\cdot)}$. Also, let $b^{\star}(\cdot):=b(X^{\star}(\cdot),u^{\star}(\cdot))$.
Using  the definition of superdifferential, 
\begin{align}
\label{eq:deriv-v-3}
\frac{\d}{\d t} e^{-r t}v(X^{\star}(t)) &=  \lim_{h\rightarrow 0^{+}}  \frac{e^{-r (t+h)}v(X^{\star}(t+h))-e^{-rt}v(X^{\star} (t))}{h} \nonumber \\ 
&=e^{-r t} \lim_{h\rightarrow 0^{+}}   \frac{e^{-r h}v(X^{\star}(t+h))- e^{-r h}v(X^{\star} (t)) + e^{-r h}v(X^{\star}(t))-v(X^{\star}(t))}{h} \nonumber \\
& = e^{-r t} \bigg(\lim_{h\rightarrow 0^{+}} e^{-r h}\frac{v(X^{\star}(t+h))- v(X^{\star} (t))}{h} -rv(X^{\star}(t))\bigg) \nonumber \\
&= e^{-r t} \bigg(\lim_{h\rightarrow 0^{+}} e^{-r h}\frac{v(X^{\star}(t) + h b^{\star}(t) + o(h))- v(X^{\star} (t))}{h} - rv(X^{\star}(t))\bigg) \nonumber \\
&\leq e^{-rt}\left(- rv( X^{\star}(t))+  \langle b^{\star}(t), (\partial^{\pm}_{s}v(X^{\star}(t)), \underline{p}_{i}(X^{\star}(t))) \rangle \right),
\end{align}
so that by \eqref{eq:limitHJB} 
\begin{align}
e^{-rt}\left(- rv( X^{\star}(t))+  \langle b^{\star}(t), (\partial^{\pm}_{s}v(X^{\star}(t)), \underline{p}_{i}(X^{\star}(t)))\rangle \right) 
\leq -e^{-rt}C(X^{\star}(t),u^{\star}(t)). \nonumber
\end{align}
Therefore
\begin{equation}\label{eq:deriv-v-3bis}
\frac{\d}{\d t} e^{-r t}v(X^{\star}(t)) \leq -e^{-rt}C(X^{\star}(t),u^{\star}(t)).
\end{equation}
Because of the locally Lipschitz property of $t \mapsto v(X^{\star}(t))$, 
we can write
\begin{equation}\label{gaf}
e^{-rT}v(X^{\star}(T))-v(x)=\int_{0}^{T}\frac{\d}{\d t} e^{-r t}v(X^{\star}(t)) \d t,
\end{equation}
which, together with \eqref{gaf} and the assumed nonnegativity of $v$, yields
\begin{align*}
&v(x)\geq e^{-rT}v(X^{\star}(T)) + \int_{0}^{T}e^{{-rt}}C(X^{\star}(t), u^{\star}(t)) \d t \geq \int_{0}^{T}e^{{-rt}}C(X^{\star}(t), u^{\star}(t)) \d t.
\end{align*}
Taking limits as $T\uparrow \infty$ and using the monotone convergence theorem by nonnegativity of $C$, we finally obtain
$$v(x) \geq \int_{0}^{\infty}e^{{-rt}}C(X^{\star}(t), u^{\star}(t)) \d t.$$
Hence, $v\geq V$ on $\mathcal{M}$, by arbitrariness of $x$. 
\end{proof}
Combining the results of the previous theorem, we see that if both the assumptions of Part (1) and Part (2) are satisfied, then  we find that $v(x)=V(x)$ and  $u^{\star}(\cdot)$ is optimal  starting at $x$.

%
%

Part (2) of Theorem \ref{thm:verification} assumes \eqref{eq:OC}, which in turn holds if a solution to the closed-loop  differential inclusion
\begin{equation}
\label{eq:closedloop-tris}
\begin{cases}
S'(t)\in -\beta S(t)I(t) - S(t)\,U^{\star}\left(S(t),I(t),\partial^{*}_{s}V (S(t),I(t))\right)+\eta (1-S(t)-I(t)),\\
I'(t)= \beta S(t)I(t) -\gamma I(t),\\
\end{cases}
\end{equation}
exists. In order to address this aspect, we consider the closed-loop equation 
\begin{equation}
\label{eq:closedloop-bis}
\begin{cases}
S'(t)= -\beta S(t)I(t) -  S(t)\, U^{\star}\left(S(t),I(t),\partial^{+}_{s}v (S(t),I(t))\right)+\eta (1-S(t)-I(t)),\\
I'(t)= \beta S(t)I(t) -\gamma I(t),\\
\end{cases}
\end{equation}
and
look at it from the point of view of the theory of \emph{Regular Lagrangian Flows}; cf.\ \cite{Ambrosio1} and \cite{Ambrosio2}, among others. For this, we recall the following definition (cf.\ Definition 1 in \cite{Ambrosio1}):

\begin{df}
\label{def:RLF}
Let $d\geq1$, denote by $\mathcal{L}^d$ the Lebesgue measure on $\mathbb{R}^d$, and fix $T>0$. We say that $X(t;x)$ is a Regular Lagrangian Flow associated to a vector field $m:\R^{d}\to\R^{d}$ if:
\begin{enumerate}
\item For $\mathcal{L}^d$-a.e.\ $x \in \mathbb{R}^d$, $t \mapsto X(t;x)$ is an absolutely continuous solution on $[0,T]$ to the ODE
$$\frac{\d}{\d t} X(t;x) = m(X(t;x)), \qquad X(0;x)=x;$$

\item for some constant $C>0$, 
$$\mathcal{L}^d\big(x\in \mathbb{R}^d:\, X(t; x) \in B\big) \leq C\cdot \mathcal{L}^d\big(B\big) \ \ \ \forall t \in [0,T], \  \forall B \subset \mathbb{R}^d \ \mbox{Borel set}.$$
\end{enumerate}

We say that a Regular Lagrangian Flow is unique if, given $X(t;x)$ and $\widetilde{X}(t;x)$ Regular Lagrangian Flows starting from $\mathcal{L}^d$-measurable sets $B_i \subset \mathbb{R}^d$, $i=1,2$, we have that
$$X(t;x) = \widetilde{X}(t;x), \qquad \text{for}\,\,\mathcal{L}^d-\text{a.e.}\, x \in B_1\cap B_2.$$
\end{df}

\begin{proposition}
\label{prop:closedloop} 
Let $v:\mathcal{M} \to \mathbb{R}_+$ be  a semiconcave, bounded,  and nonnegative viscosity solution to \eqref{eq:HJB}. 
Set $$\widetilde{U}^{\star}(s,i):=U^{\star}\left(s,i,\partial^{+}_{s}v (s,i)\right)$$
and assume 
 that 
 \begin{itemize}
 \item[(i)]
 $(s,i,p_s) \mapsto {U}^{\star}(s,i,p_s)$ is locally Lipschitz-continuous on $\mathcal{M} \times \mathbb{R}$;  
\item[(ii)]
$(s \frac{\partial\widetilde{U}^{\star}}{\partial s})^{+}\in L^{\infty}(\mathcal{M})$. 
\end{itemize}
Then, the Regular Lagrangian Flow (closed-loop equation)
\begin{equation}
\label{eq:closedloop-tris}
\begin{cases}
S'(t)=-\beta S(t)I(t) - \widetilde{U}^{\star}\left(S(t),I(t)\right) S(t)+\eta (1-S(t)-I(t)),\\
I'(t)= \beta S(t)I(t) -\gamma I(t),\\
\end{cases}
\end{equation}
with initial data $(S(0),I(0))=(s,i) \in \mathcal{M}$ exists and is unique.
\end{proposition}
\begin{proof}
It is enough to embed the closed-loop equation \eqref{eq:closedloop-tris} in the setting of Definition \ref{def:RLF}, and apply Theorem 7 in \cite{Ambrosio1} (see also Theorems 6.2 and 6.4 in \cite{Ambrosio2}), after checking the validity of its hypothesis.

In order to accomplish that, we set 
$$x=(x_1,x_2)=:(s,i) \in \mathcal{M},  \ \ \ \ X(t;x)=(X_1,X_2)(t;x)=:(S(t),I(t))$$ and 
$$m(x):=\Big(-\beta x_1 x_2 - x_{1}\widetilde{U}^{\star}(x_1,x_2)+\eta(1-x_1-x_2), \ \  \beta x_1 x_2 - \gamma x_2\Big), \ \ \ \ x\in\mathcal{M},$$ 
(being extended to $\R^{2}$ by defining it equal to $(0,0)$ on $\mathcal{M}^{c}$).

Firstly, since $v$ is semiconcave by Theorem \ref{thm:viscosity}, it follows by Theorem 3 at page 240 of \cite{EvansGariepy} that $\partial^{+}_{s}v$ is locally of bounded variation on $\mathcal{M}$. Then, as $(s,i,p_s) \mapsto {U}^{\star}(s,i,p_s)$ is locally Lipschitz-continuous by assumption, it follows that $\widetilde{U}^{\star}$ is locally of bounded variation on $\mathcal{M}$ too by Theorem 4 in \cite{Josephy}, so that $m$ is locally of bounded variation on $\R^2$.

Secondly, given the fact that $(x_1,x_2) \in \mathcal{M}$ and $\widetilde{U}^{\star} \in [0, \overline{U}]$, it is readily seen that $\frac{|m(x)|}{1 + |x|}$ is integrable over $\R^{2}$. Furthermore, explicit computations yield in $\mathcal{M}$, for some $K\geq 0$,
$$\big(\text{div}\,m\big)^{-} \leq K +\Big(- x_1 \frac{\partial \widetilde{U}^{\star}}{\partial x_1}\Big)^{-}= K + \Big(x_{1}\frac{\partial \widetilde{U}^{\star}}{\partial x_1}\Big)^{+}.$$
Since $ \big(s\frac{\partial\widetilde{U}^{\star}}{\partial s}\big)^{+} \in L^{\infty}(\mathcal{M})$ by assumption, we conclude.
\end{proof}

\begin{remark}
\label{rem:RLF}
The conditions of Proposition \ref{prop:closedloop}
can be  verified on a case by case basis. For example, they hold in the quadratic example considered in the next section. Let $$
C(s,i,u) =\frac{1}{2} \left(a {i}^2 +b (us)^2 \right), \ \ \ a,b>0.
$$
Then, given $v$ as in Proposition \ref{prop:closedloop}, we have
$\widetilde{U}^\star(s,i)= U^{\star}(s,i,\partial^{+}_{s}v(s,i))$, where

\begin{eqnarray*}
U^{\star}(s,i,p_{s}) &=& 
\left\{\begin{array}{lll}
0 && \textrm{if } p_{s}\le 0,\\\\ 
\frac{p_{s}}{bs} && \textrm{if }  0<p_{s} <b \overline{U}s,\\ \\
\overline{U} && \textrm{if } p_{s}\ge b \overline{U}s,
\end{array}\right.
\end{eqnarray*}
Now,  $(s,i,p_s) \mapsto {U}^{\star}(s,i,p_s)$ is clearly locally Lipschitz continuous. Moreover,
%
%
we notice that
$$
\frac{\partial{U}^{\star}}{\partial s}(s,i,p_s)=\begin{cases}
0, \ \ \ \ \ \ \ \ \ \ \  \mbox{if} \ p_{s}\notin (0,b \overline{U}s),\\
-\frac{p_{s}}{bs^{2}},  \ \ \ \ \ \ \mbox{if} \ \ p_{s}\in (0,b \overline{U}s), \ \ \ 
\end{cases}
 \ \ \ \ \ \frac{\partial{U}^{\star}}{\partial p_s}(s,i,p_s)=\begin{cases}
0, \ \ \ \ \ \ \ \ \ \ \mbox{if} \ p_{s}\notin (0,b \overline{U}s),\\
\frac{1}{bs},  \ \ \ \ \mbox{if} \ \ p_{s}\in (0,b \overline{U}s),
\end{cases}
$$ 
Therefore, $ \frac{\partial {U}^{\star}}{\partial s}\leq 0$ and we have (in the sense of distributions)
\begin{eqnarray*}
s\frac{\partial \widetilde{U}^{\star}}{\partial s}(s,i)= s\frac{\partial{{U}}^{\star}}{\partial s} (s,i, \partial^{+}_{s}v(s,i))+ s \frac{\partial{{U}}^{\star}}{\partial p_s}(s,i, \partial^{+}_{s}v(s,i))\cdot  \frac{\partial}{\partial s}\big(\partial^{+}_{s}v\big)(s,i)
\leq  \frac{1}{b} \frac{\partial}{\partial s}\big(\partial^{+}_{s}v\big)(s,i).
\end{eqnarray*}
Then, the semiconcavity of $v$ allows to conclude $(s \frac{\partial\widetilde{U}^{\star}}{\partial s})^{+}\in L^{\infty}(\mathcal{M})$. 
\end{remark}
As a corollary of Theorems \ref{thm:viscosity} and \ref{thm:verification}, and of Proposition \ref{prop:closedloop}, we obtain uniqueness of viscosity solution to HJB \eqref{eq:HJB}.
\begin{corollary} Let Assumption \ref{ass:C} and the assumptions of Proposition \ref{prop:closedloop} hold. Then $V$ is the unique bounded  locally semiconcave viscosity solution to \eqref{eq:HJB}.
\end{corollary}
\begin{proof}
We need to prove uniqueness. By Proposition \ref{prop:closedloop}, a solution to \eqref{eq:closedloop-bis} exists for almost every $(s,i)\in\mathcal{M}$. This provides the control $u^{\star}$ required by Part (1) of Theorem  \ref{thm:verification}. Hence, given any other bounded locally semiconcave viscosity solution $v$ to \eqref{eq:HJB}, we get $v(s,i)=V(s,i)$ for a.e. $(s,i)\in\mathcal{M}$. Since semiconcavity implies continuity, we conclude that $v=V$ on $\mathcal{M}$. 
\end{proof}


\section{A Case Study with Numerical Illustrations}
\label{sec:numerics}

In this section we introduce a case study and provide numerical simulations in order to illustrate the results of the proposed model.
We preliminary consider as benchmark the dynamical system \eqref{eq:syst-dyn-S}-\eqref{eq:syst-dyn-I} in absence of vaccination. We assume that $\mathcal{R}_{o}=\beta/\gamma>1$. In this case the dynamical system has two equilibria:
\begin{equation}\label{equilibriano}
(S^{(1,o)}_{\infty},I^{(1,o)}_{\infty})=\left(\frac{\gamma}{\beta}, \frac{\eta}{\gamma+\eta}\left(1-\frac{\gamma}{\beta}\right)\right), \ \ \ \ (S^{(2,o)}_{\infty},I^{(2,o)}_{\infty})=(1,0).
\end{equation}
The convergence of the system to the first of the above equilibria means that the disease becomes endogenous; the convergence to the second one means that the disease goes to extinction. It is shown in \cite{ORegan-etal} that $(S^{(1,o)}_{\infty},I^{(1,o)}_{\infty})$ is globally asymptotically stable when $\mathcal{R}_o=\beta/\gamma >1$, whereas $(S^{(2,o)}_{\infty},I^{(2,o)}_{\infty})$ is globally asymptotically stable when $\mathcal{R}_o=\beta/\gamma<1$.

We assume that the cost function has the following quadratic form
$$
C(s,i,u) =\frac{1}{2} \left(a {i}^2 +b (us)^2 \right), \ \ \ a,b>0.
$$
The latter can be interpreted as a second-order Taylor approximation of any smooth, convex, separable cost function with global minimum in $(0,0)$.
The parameter $a$ can be taken, for instance, as $a=\bar{\iota}\,^{{-1}}$, where $\bar{\iota}\in (0,1)$  represents the maximal percentage
of infected people that the health-care system can handle; on the other hand, $b$ represents the sensitivity of the policy maker with respect to the vaccination costs. Under this specification of the cost function, for each $(s,i) \in \mathcal M$, we have with the notation of Proposition \ref{prop:closedloop},
\begin{eqnarray*}
\widetilde{U}^\star(s,i)&=& 
\left\{\begin{array}{lll}
0 && \textrm{if } \partial^{+}_{s}V(s,i)\le 0,\\\\ 
\frac{\partial^{+}_{s}V(s,i)}{bs} && \textrm{if }  0<\partial^{+}_{s}V(s,i) <b \overline{U}s,\\ \\
\overline{U} && \textrm{if } \partial^{+}_{s}V(s,i)  \ge b \overline{U}s.
\end{array}\right.
\end{eqnarray*}

Our numerical method is based on a recursion on the HJB equation \eqref{eq:HJB}. 
Precisely, starting from $v^{[0]} \equiv 0$, we use the recursive
algorithm:
$$
(r-\mathcal L)v^{[n+1]} = C^{\star}(s,i,v^{[n]}_{s}), \quad n \ge 0.
$$
Those equations are then solved using the representation formula
$$
v^{[n+1]}(s,i)=\int_0^\infty e^{-rt} C^\star\big(S_t^{s,i},I_t^{s,i},v^{[n]}_{s}(S_t^{s,i}, I_t^{s,i})\big) dt, \quad (s,i) \in \mathcal M.  
$$
The evaluation of the needed derivatives is performed through a finite difference scheme, which is of backward or forward type for those points that lie at the boundary of the region $\mathcal{M}$.

Throughout this section, the unit of time will be the week. 

\subsection{Optimal vaccination vs.\ no vaccination}
\label{OptimalVsNo} 

We set the following values for the parameters. 
The transmission rate of the disease is $\beta=0.7$;
the average length of infection and reinfection are assumed to be equal to $21$ and $180$ days respectively, so that $\gamma =\frac{1}{3}$ and $\eta=\frac{7}{180}$;
we set $r=\frac{0.005}{52}$ (i.e.\ a yearly discount rate of $5\%$) and $a=0.08$ and $b=0.016$.
We fix $\overline{U}=\frac{7}{120}$. Since $e^{-52\overline{U}}\approx e^{-3}$, one observes that the health-care system vaccinating at the maximal rate $\overline{U}$ is able to vaccinate about $95\%$ of the total  population in $1$ year. 
Finally, we initialize the system as
$$
S(0)=75\%, \ \ \ I(0)=20\%.
$$

Figure \ref{fig1} provides a comparison between the optimal vaccination policy and the no-vaccination policy. 
\begin{figure}[htb]
\begin{center}
\begin{tabular}{cccc}
\includegraphics[height=4cm]{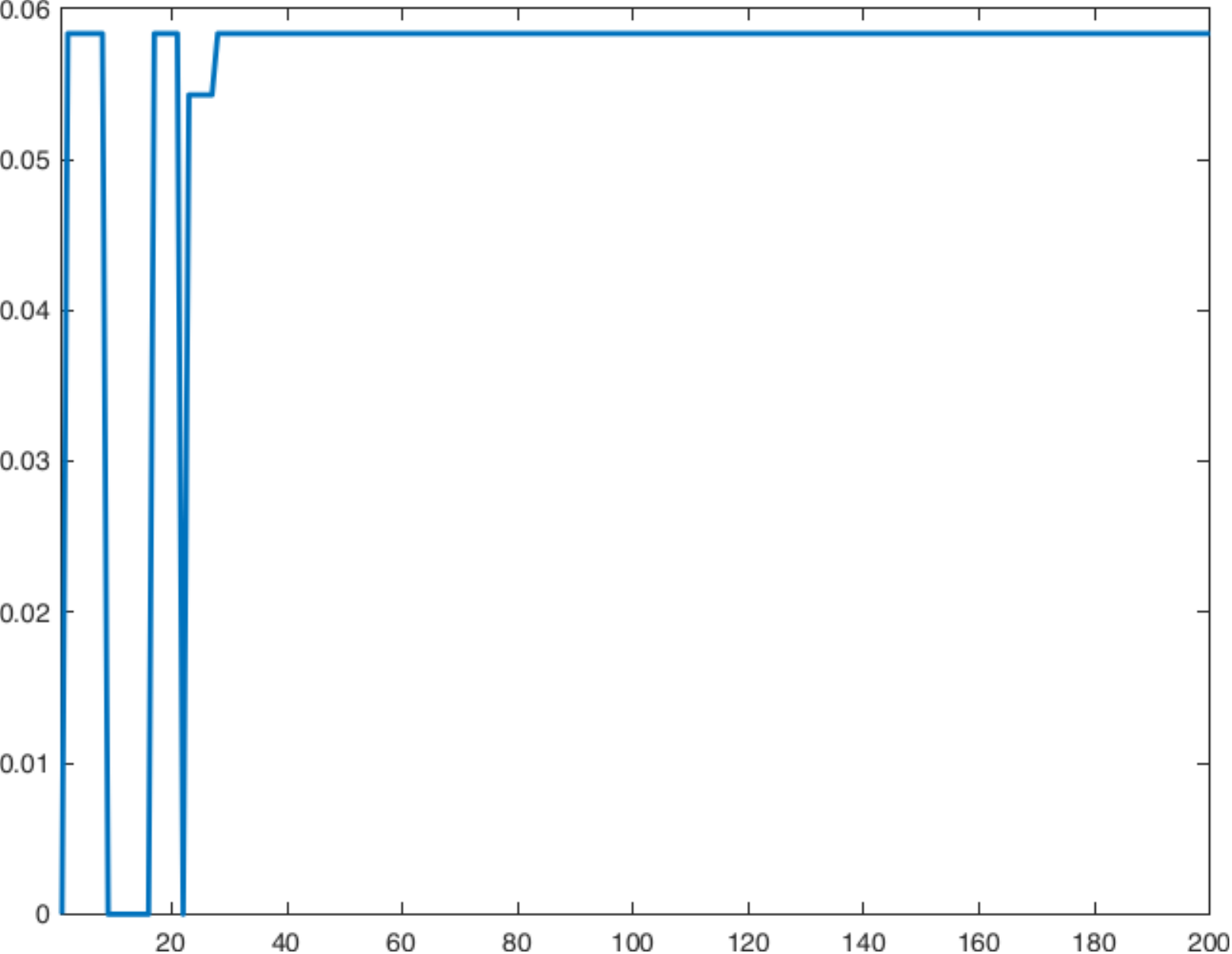} &
 \includegraphics[height=4cm]{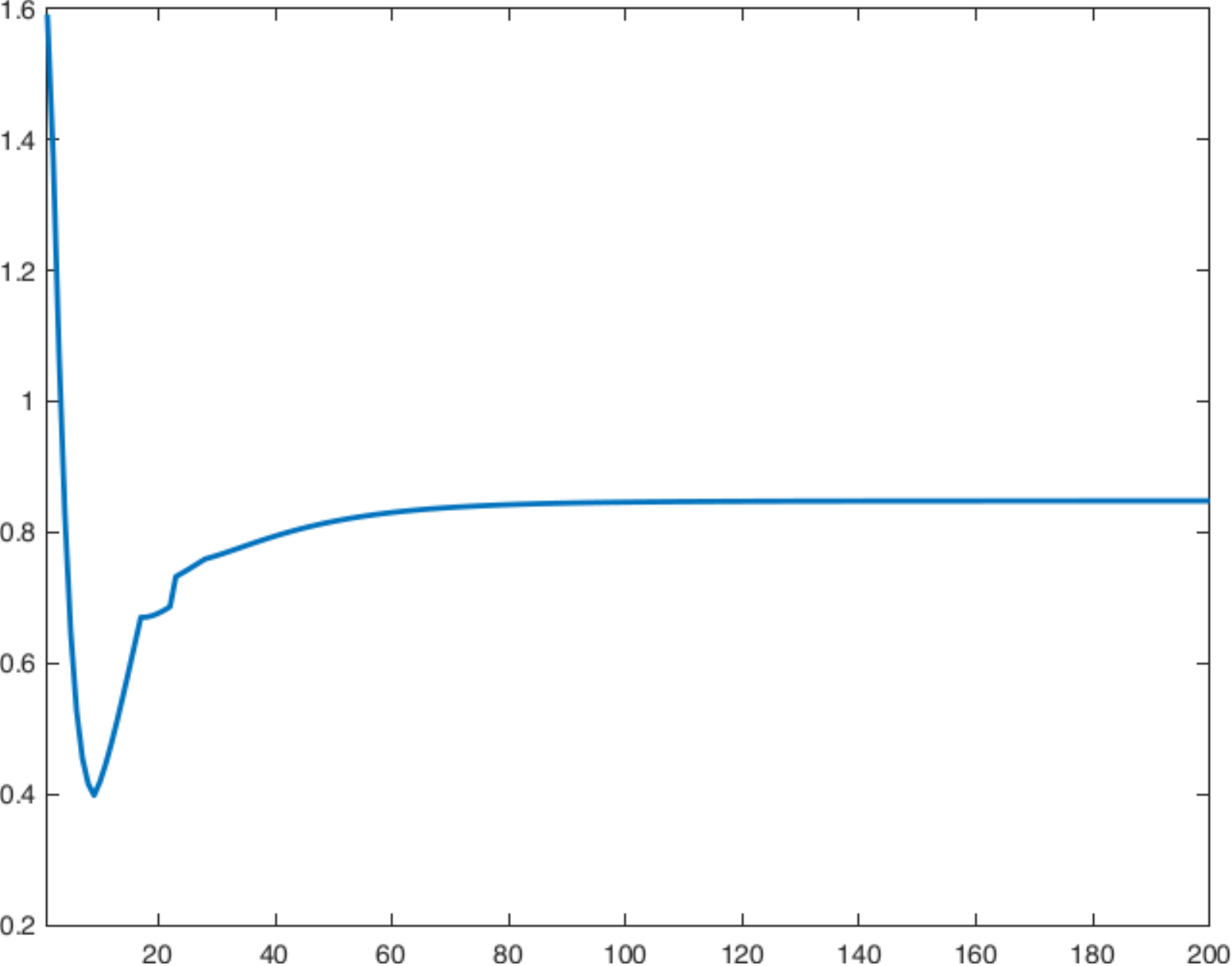} &  
 \includegraphics[height=4cm]{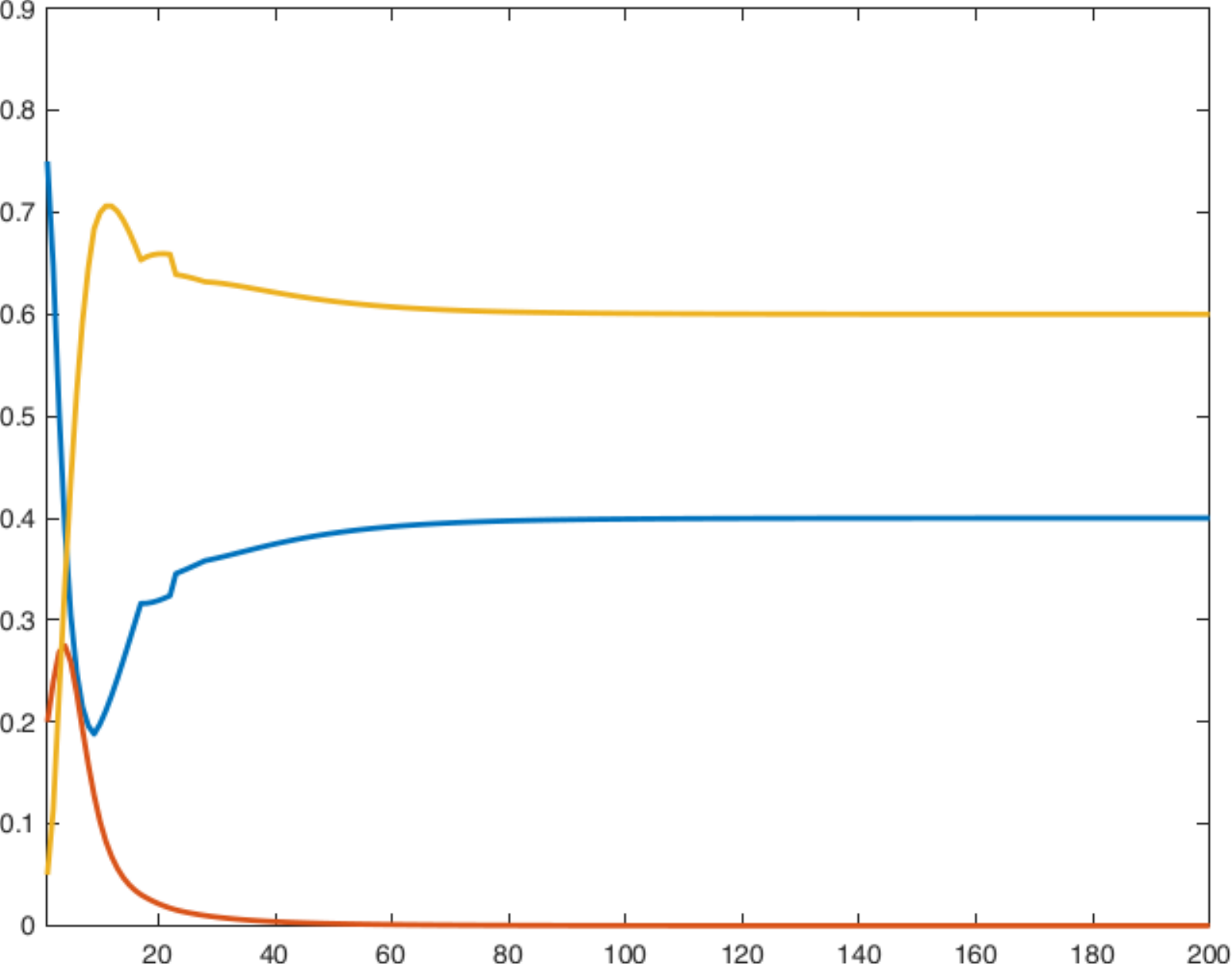}  
\\ \\
\includegraphics[height=4cm]{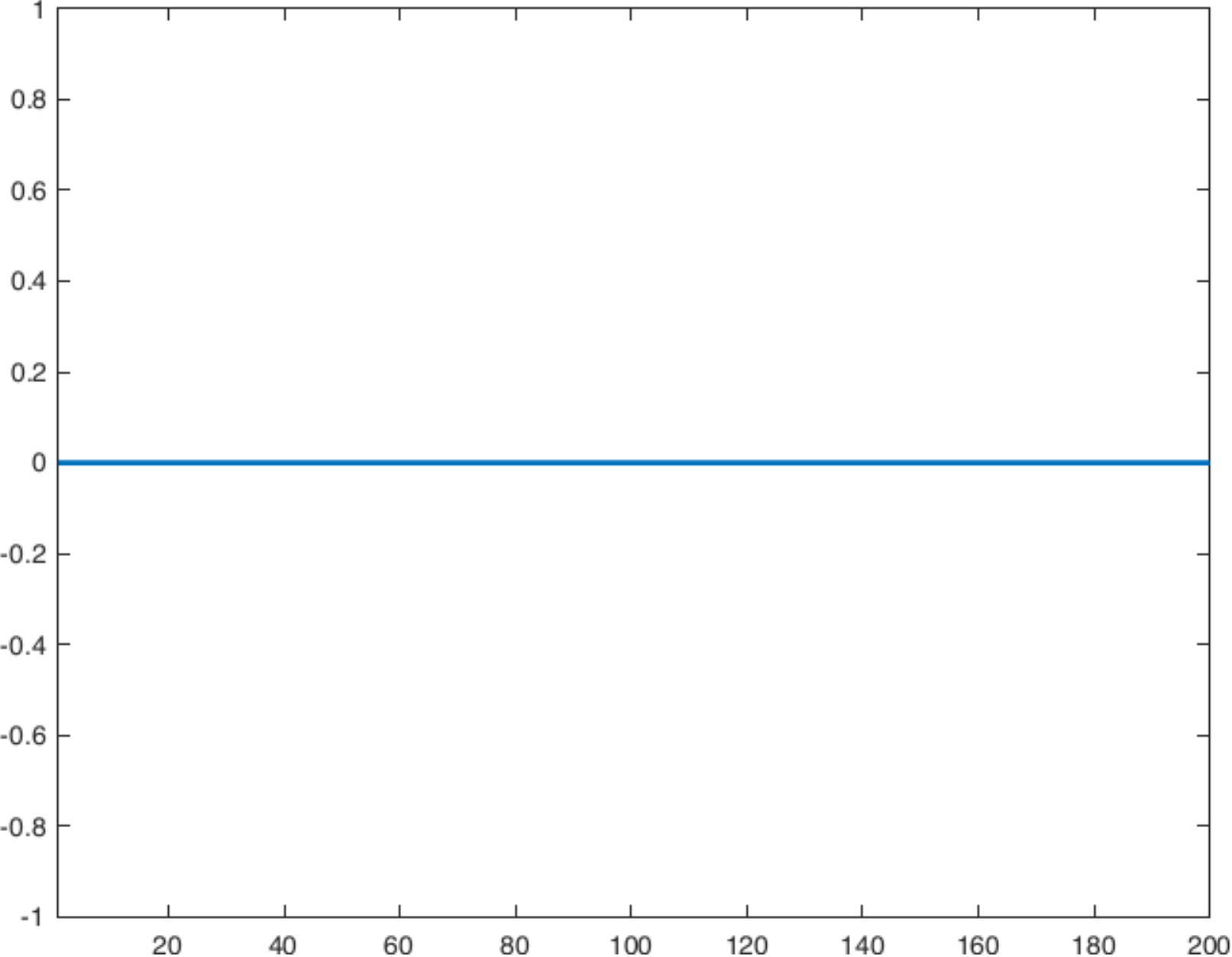}& 
\includegraphics[height=4cm]{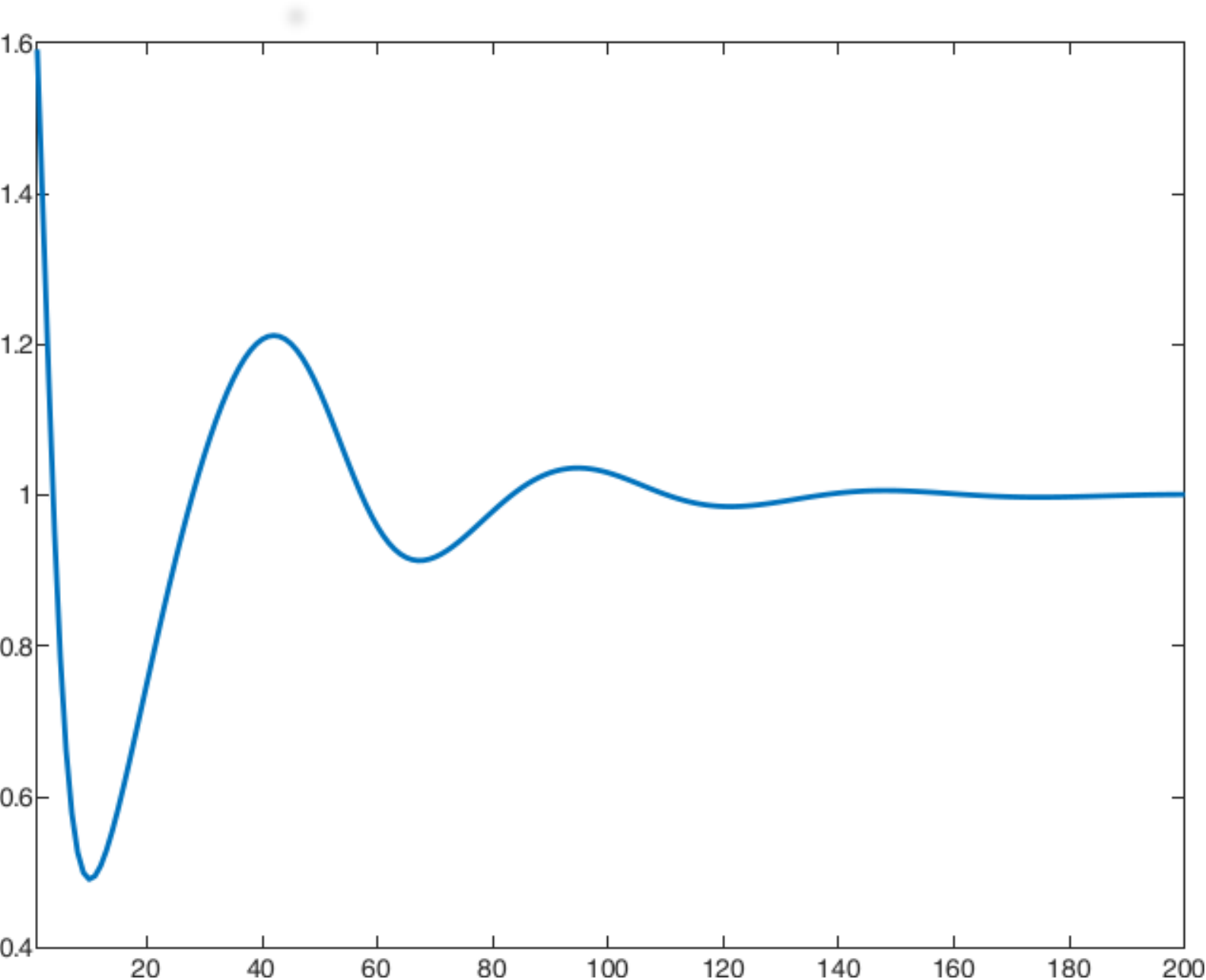} &
\includegraphics[height=4cm]{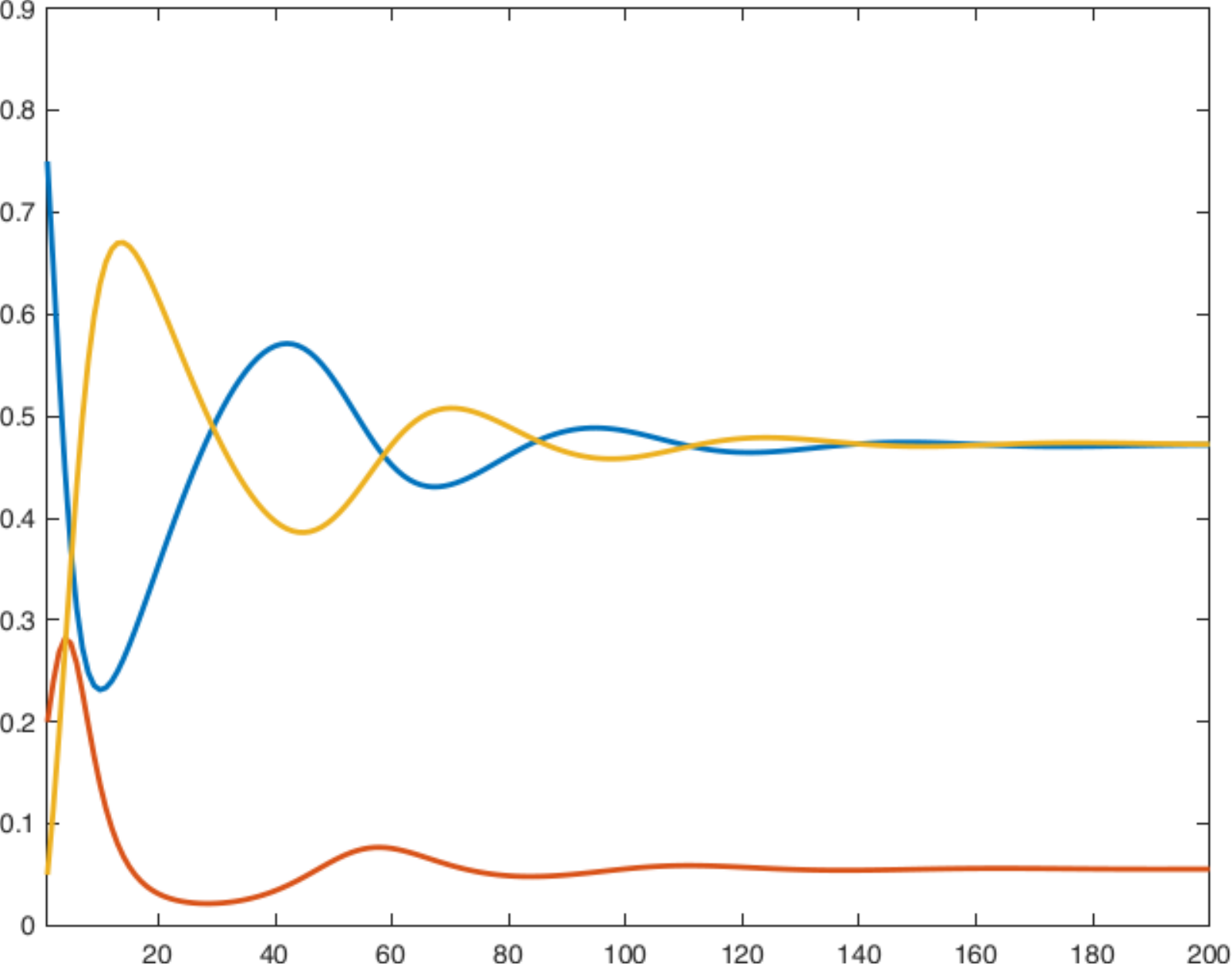}\\
\end{tabular}
\caption{Comparison between the optimal social planner vaccination policy (upper panel) and the case of no vaccination (lower panel).
The figures in the first column show the evolution of the vaccination policy through the optimal control $u_t$; the ones in the second column show  
the evolution of the instantaneous reproduction number $\mathcal{R}_t=\frac{\beta}{\gamma}S(t)$; 
the ones in the third column show evolution of the percentage of susceptible (in blue), infected (in red) and recovered (in green) individuals.}
\label{fig1}
\end{center}
\end{figure}
As outcome, we see that in both cases the system oscillates in a first phase and then converges towards an equilibrium. 

Our numerical simulations suggest that the optimal policy $u^{\star}$ always converges towards a limit value 
 $u^\star_{\infty}$. Assuming that such an equilibrium $u^\star_{\infty}$ indeed exists, 
 we see that the optimally controlled system has again two long-run equilibria  (cf.\ \eqref{equilibriano}):
$$
 (S_{\infty}^{(1)},I_{\infty}^{(1)})=\left(\frac{\gamma}{\beta}, \frac{\eta}{\eta+\gamma}\left(1-\frac{\gamma}{\beta}\right)-\frac{\gamma}{\eta+\gamma}\frac{u_{\infty}^\star}{\beta} \right)
 \quad \textrm{ e } \quad 
 (S_{\infty}^{(2)},I_{\infty}^{(2)})=\left(\frac{\eta}{\eta+u_{\infty}^\star},0\right),
 $$
whenever $I_{\infty}^{(1)}>0$. Repeating the arguments in \cite{ORegan-etal} (easily adjusted to our setting), it can be proved that $(S_{\infty}^{(1)},I_{\infty}^{(1)})$ is globally asymptotically stable when $u_{\infty}^\star < \gamma$.

We can observe that in the case of no vaccination the disease becomes endogenous, whereas the optimal vaccination is able in the long-run to keep the number of infected to zero. 
More precisely, in the case of no vaccination, the long-run percentage of susceptible and infected individuals, $47\%$ and $6\%$, respectively, is achieved in about $3$ years. On the contrast, if the policy maker adopts the optimal vaccination policy, the vaccination campaign starts with maximum intensity $\overline{U}$ and then it fluctuates for a period of about $28$ weeks. After that, it stabilizes at the value $\overline{U}$. 
In this case, the equilibrium point $(S_{\infty}^{(2)},I_{\infty}^{(2)})$ is approached counting about $40\%$ of susceptible individuals and almost no infected in less than $1$ year. Still, to avoid a new outbreak of the disease, the vaccination policy $\overline{U}$ 
must be kept to move individuals from the class $S$ to the class $R$.

\subsection{Variation of the parameter $\eta$}
\label{eta}
In this subsection we study how the optimal vaccination rate, the optimal reproduction number $\mathcal{R}_t$ and the optimally controlled dynamics of susceptible, infected and recovered depend on the reinfection rate $\eta$. We assume that the average period of reinfection is equal to $60$ or $360$ days, so that either $\eta =\frac{7}{360}$ or $\eta =\frac{7}{60}$, respectively. 
All the other parameters are instead kept fixed to the values assumed in Section \ref{OptimalVsNo}. 

As expected, from Figure \ref{fig4} we observe that the optimal vaccination rate increases when increasing $\eta$. However, the more vigorous optimal vaccination rate employed when $\eta =\frac{7}{60}$ is not such to let infections to zero. As a matter of fact, the lower row of Figure \ref{fig4} shows that the number of infected stabilizes asymptotically around the level $0.05$. On the other hand, the infected population disappears in the long run through a weaker vaccination policy when the reinfection average period is of 1 year circa.

\begin{figure}[htb]
\begin{center}
\begin{tabular}{cccc}
\includegraphics[height=4cm]{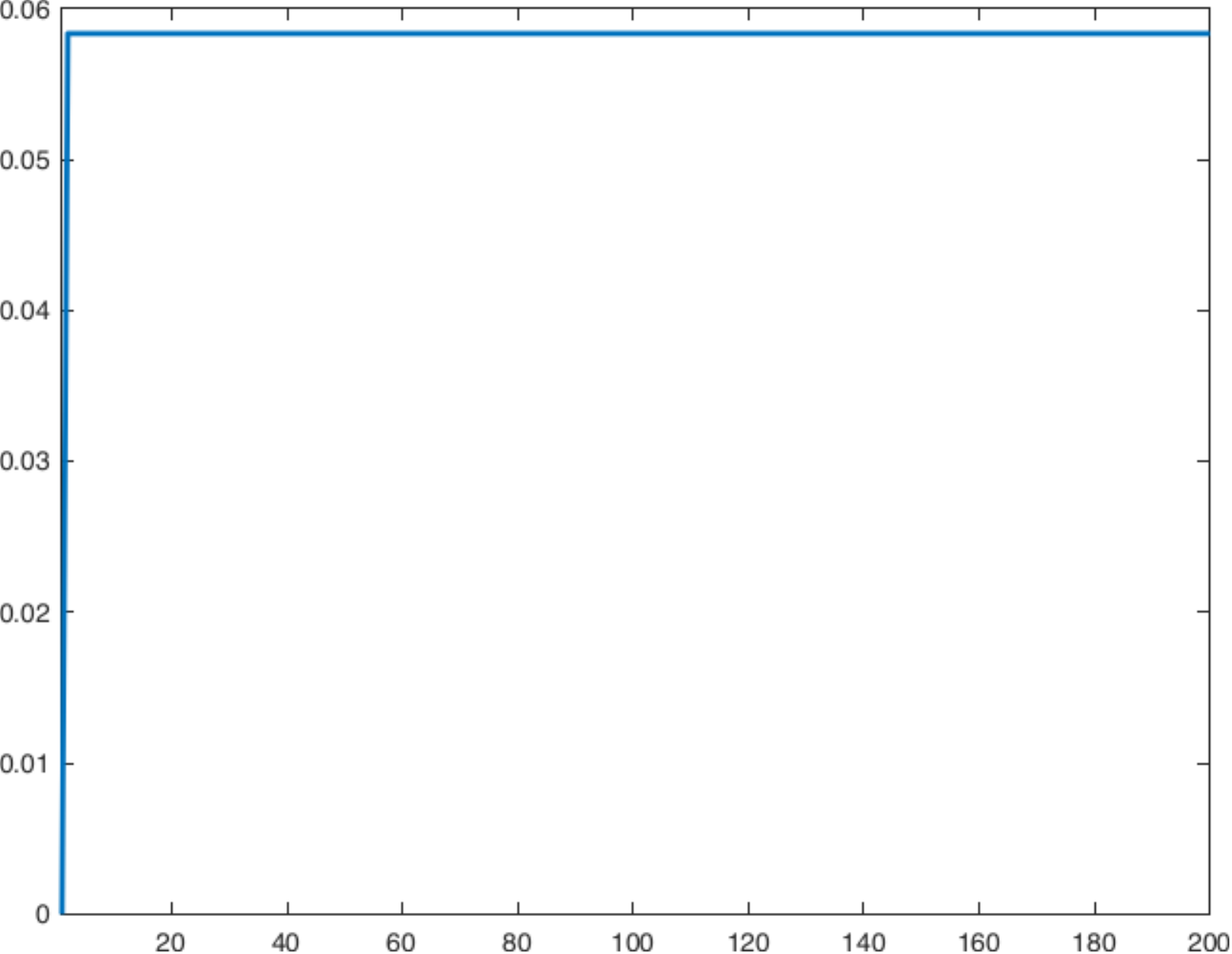} &
 \includegraphics[height=4cm]{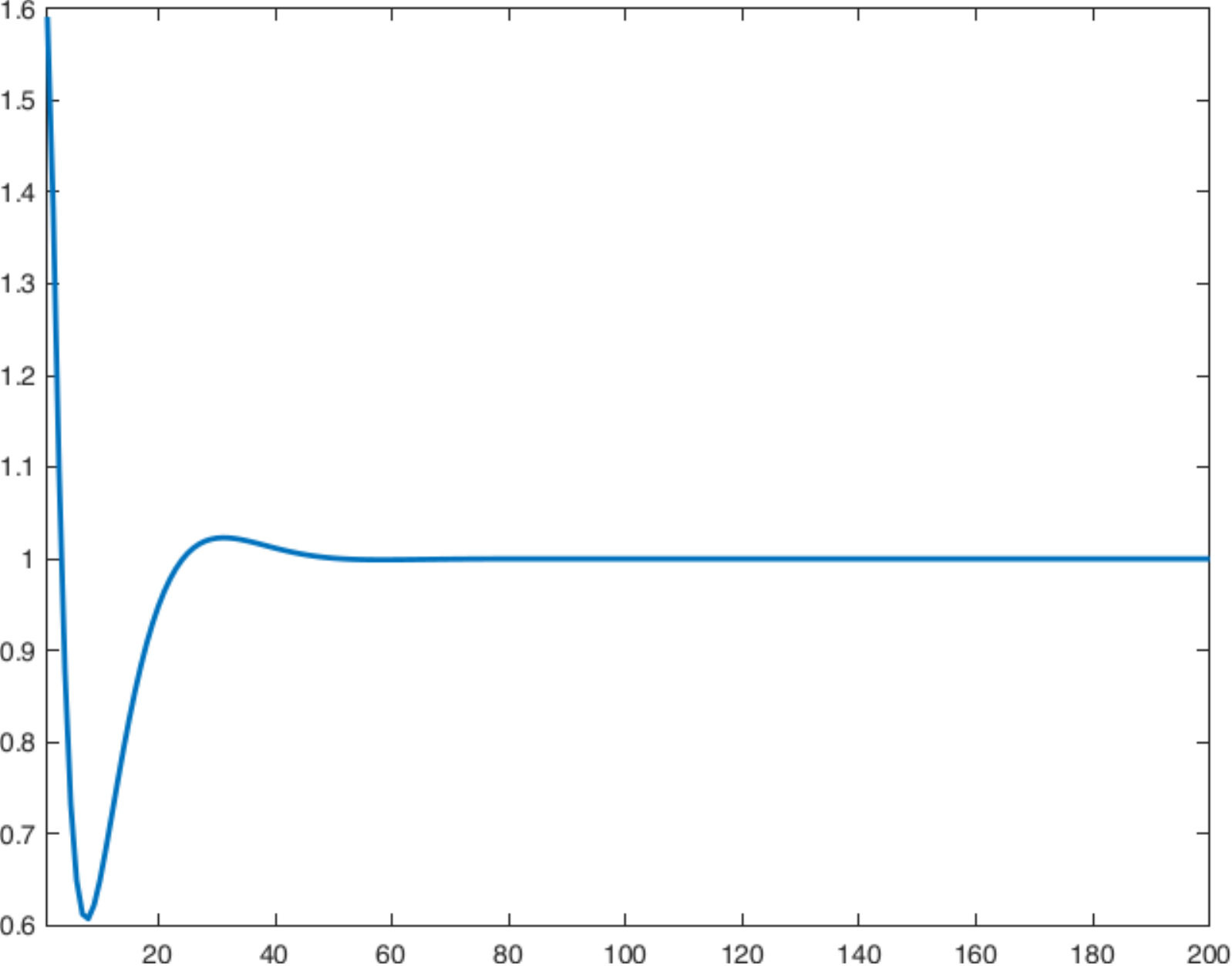} &  
 \includegraphics[height=4cm]{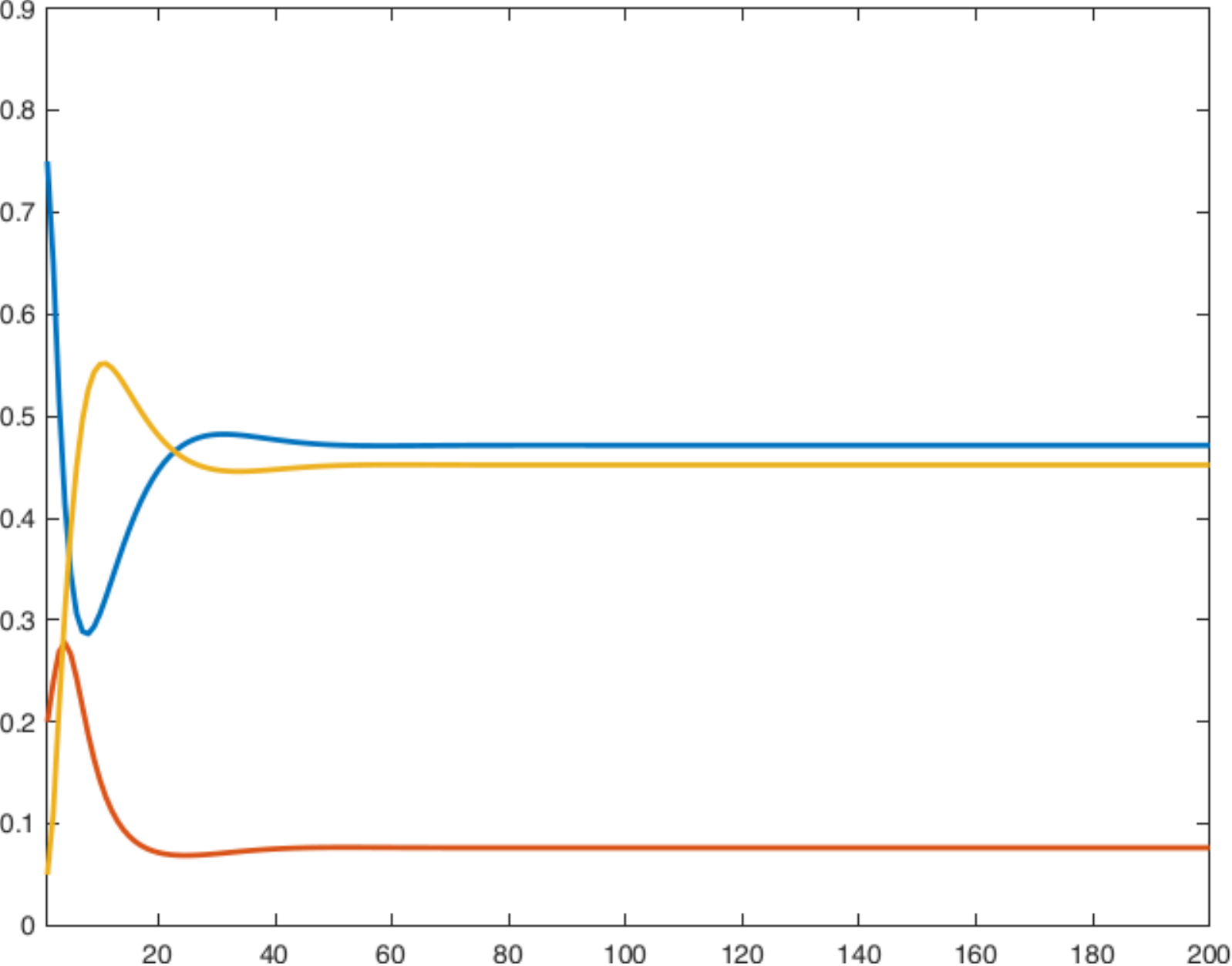} 
 \\ \\
\includegraphics[height=4cm]{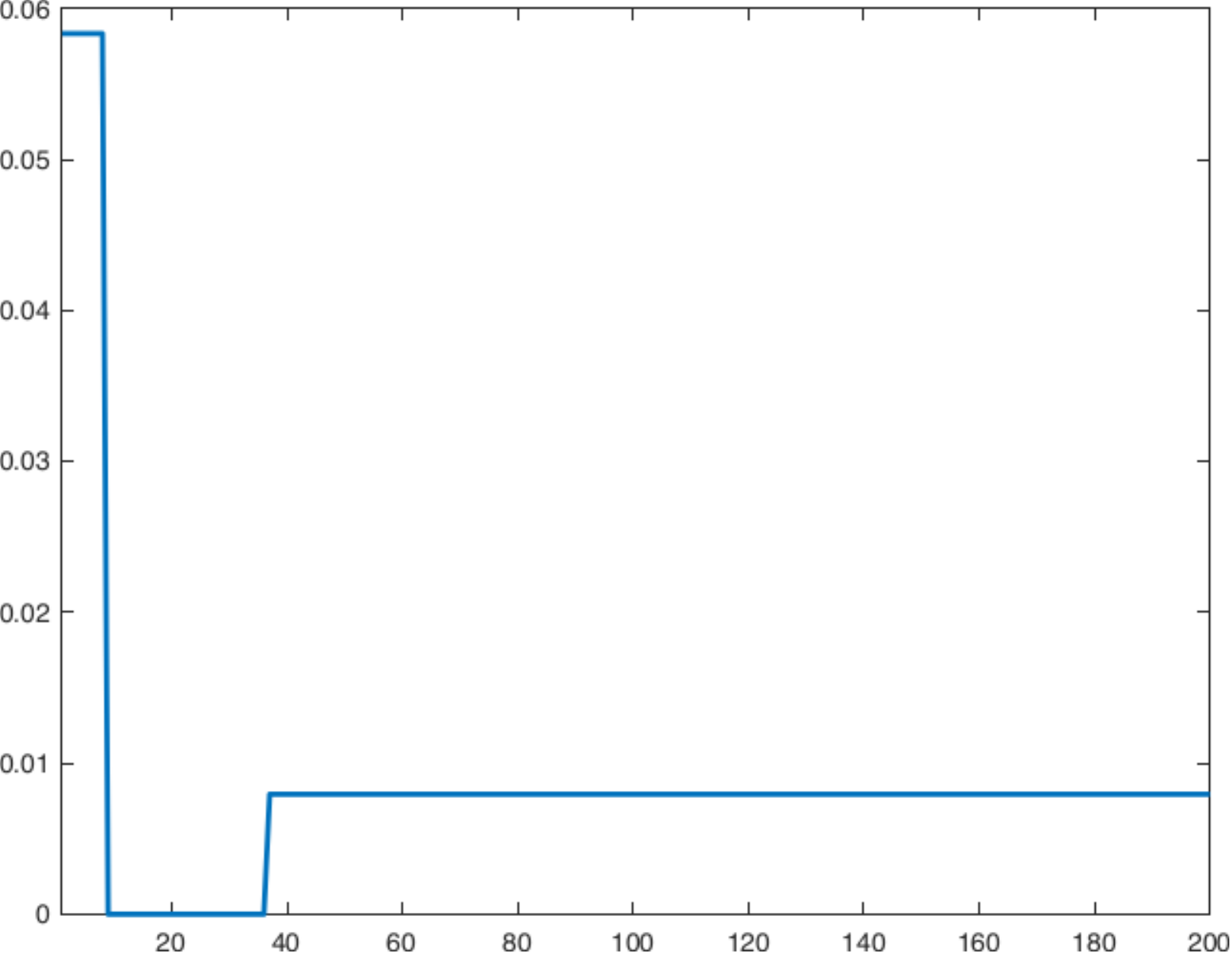} &
 \includegraphics[height=4cm]{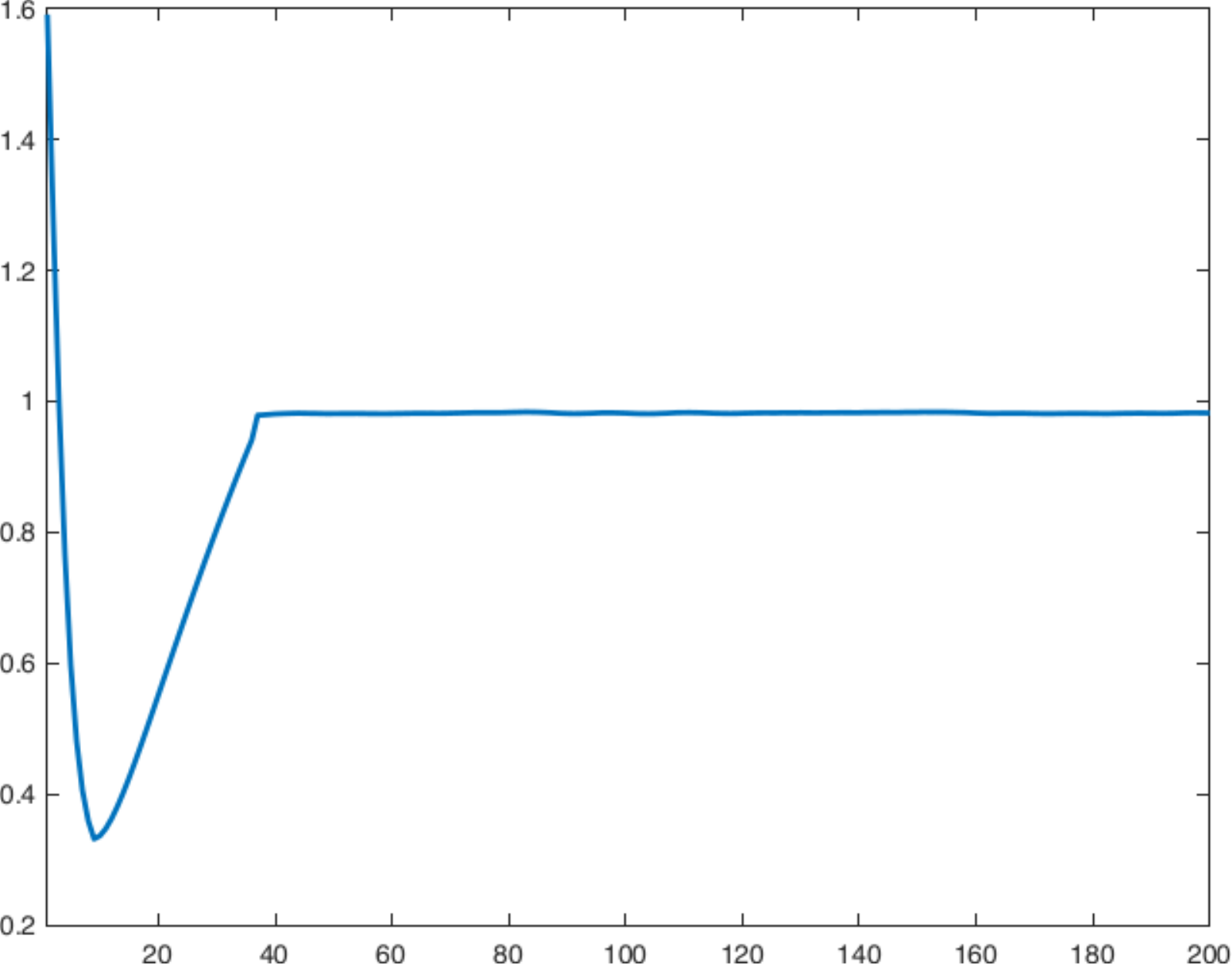} &  
 \includegraphics[height=4cm]{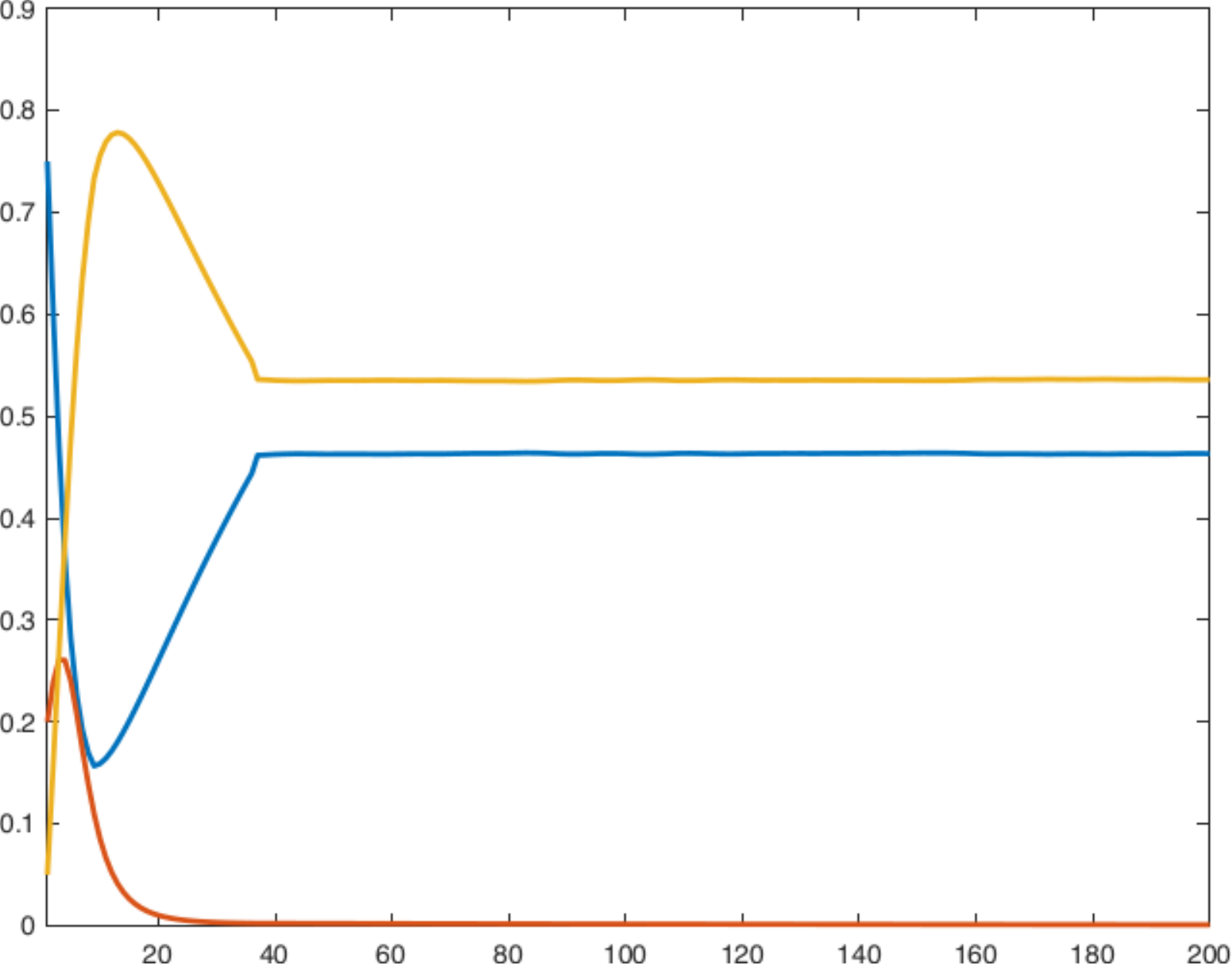} 
\end{tabular}
\caption{Comparison between the optimal social planner vaccination policy when the average length of reinfection is $60$ days (first row) or $360$ days (second row).
The figures in the first column show the evolution of the vaccination policy through the optimal control $u_t$; the ones in the second column show  
the evolution of the instantaneous reproduction number $\mathcal{R}_t=\frac{\beta}{\gamma}S(t)$; 
the ones in the third column show the evolution of the percentage of susceptible (in blue), infected (in red) and recovered (in green) individuals.}
\label{fig4}
\end{center}
\end{figure}

\subsection{Variation of the ratio $\frac{\beta}{\gamma}$}
\label{R0}

In this section we consider strategies corresponding to different values of the natural reproduction number $\mathcal{R}_o=\frac{\beta}{\gamma}$; precisely, we take $\frac{\beta}{\gamma}=3$ and  $\frac{\beta}{\gamma}=\frac{3}{2}$. All the other parameters are instead kept fixed to the values assumed in Section \ref{OptimalVsNo}. 
A comparison of the optimal social planner vaccination policy is shown in Figure \ref{fig6}.

\begin{figure}[htb]
\begin{center}
\begin{tabular}{cccc}
\includegraphics[height=4cm]{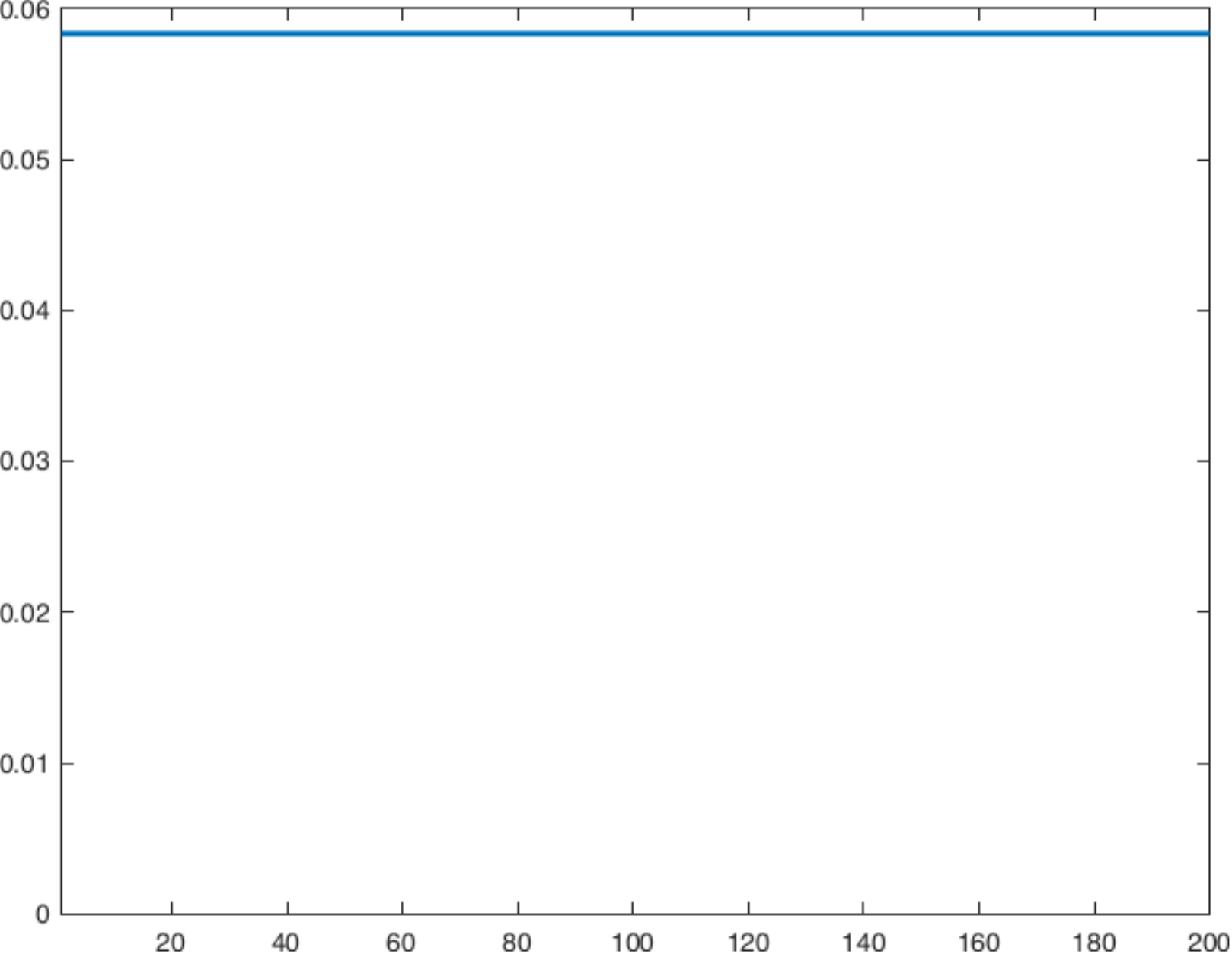} &
 \includegraphics[height=4cm]{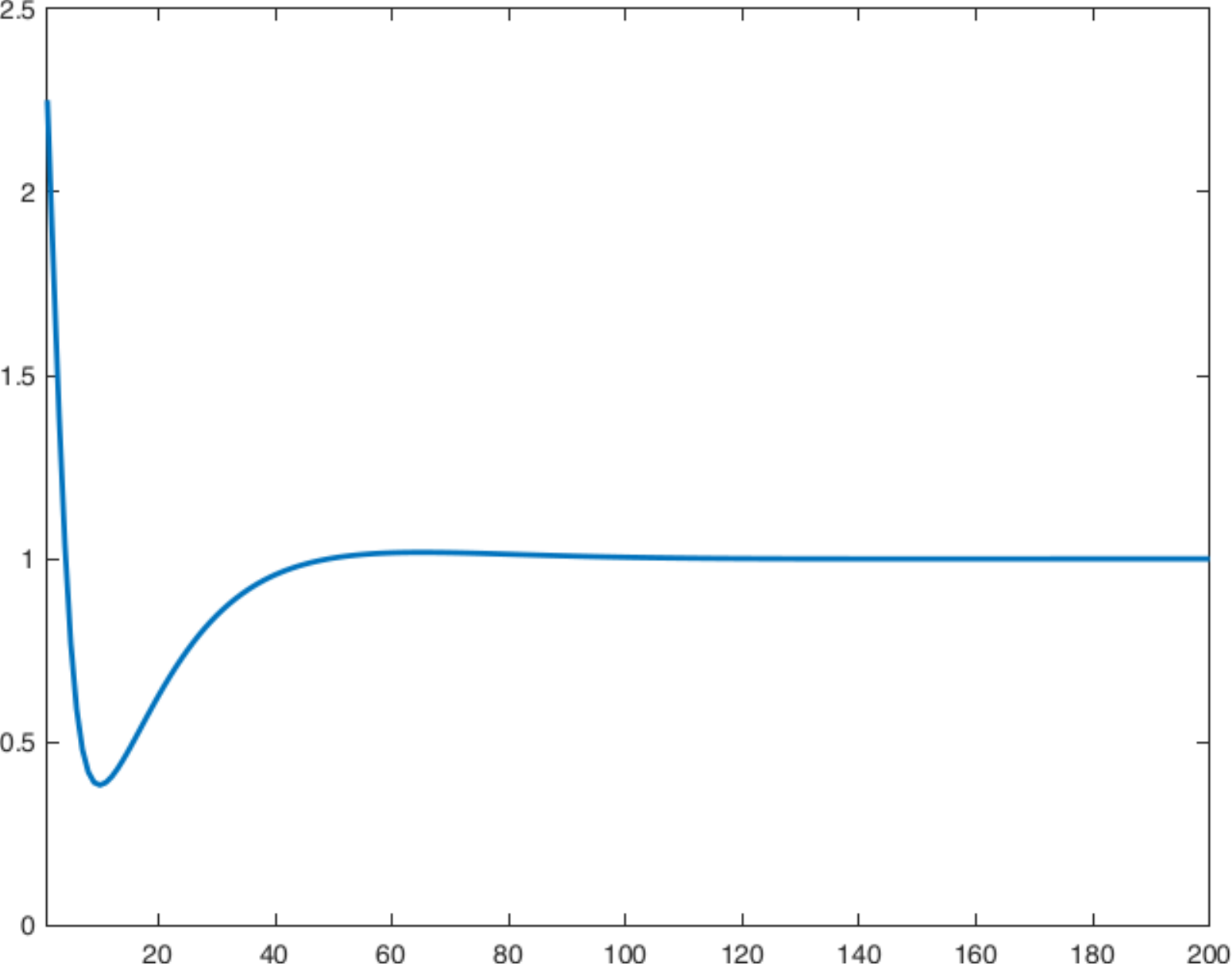} &  
 \includegraphics[height=4cm]{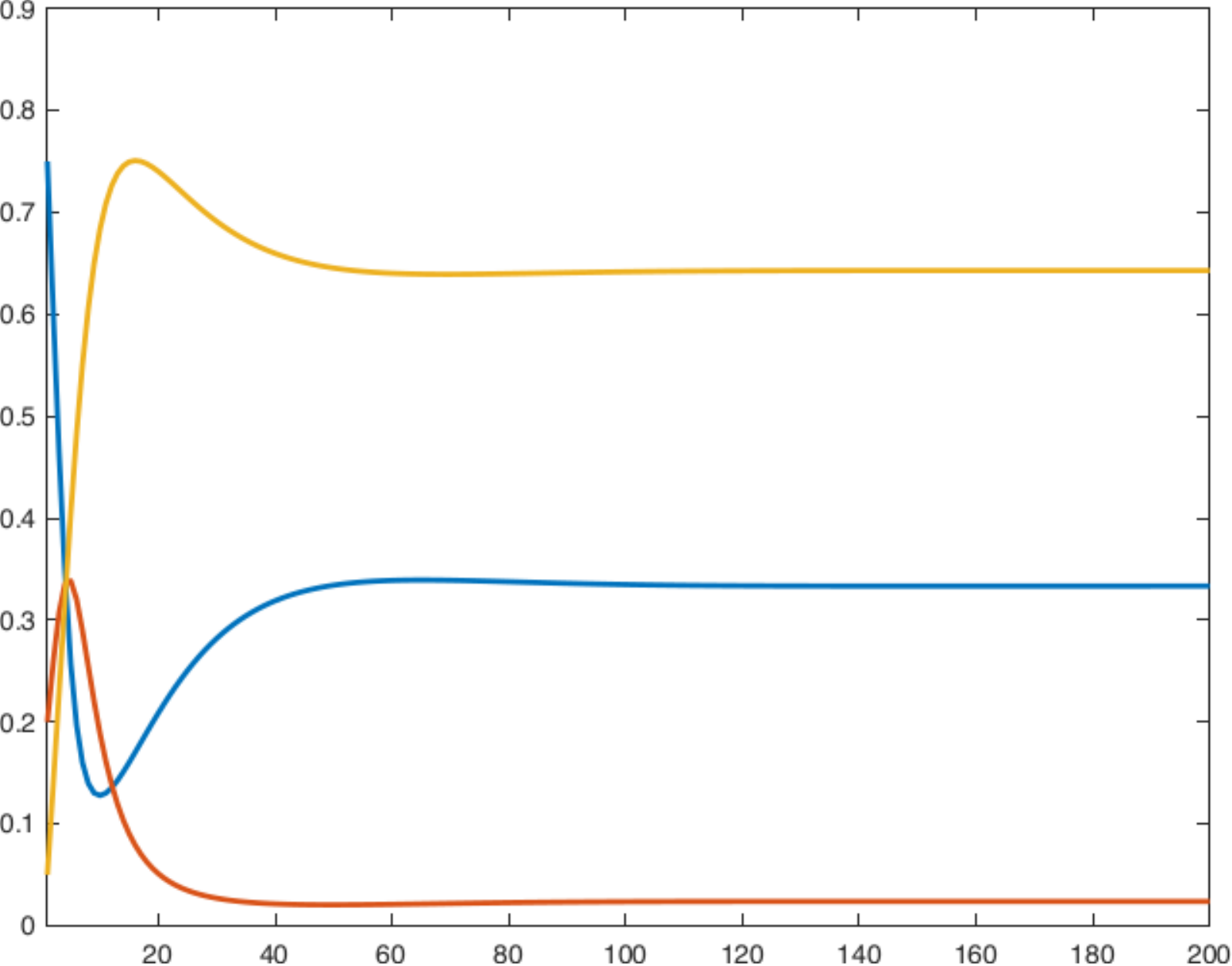}  
\\ \\
\includegraphics[height=4cm]{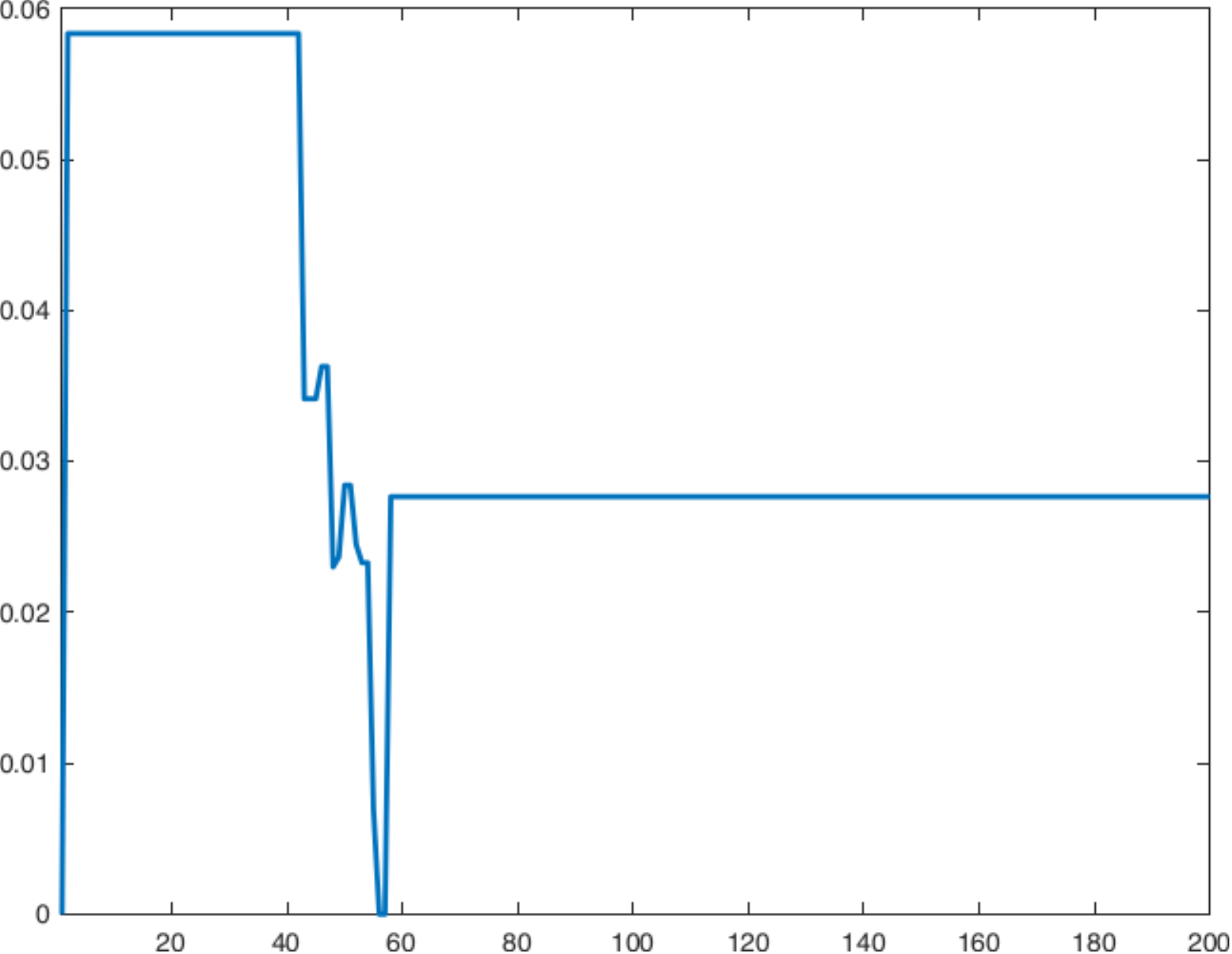} &
 \includegraphics[height=4cm]{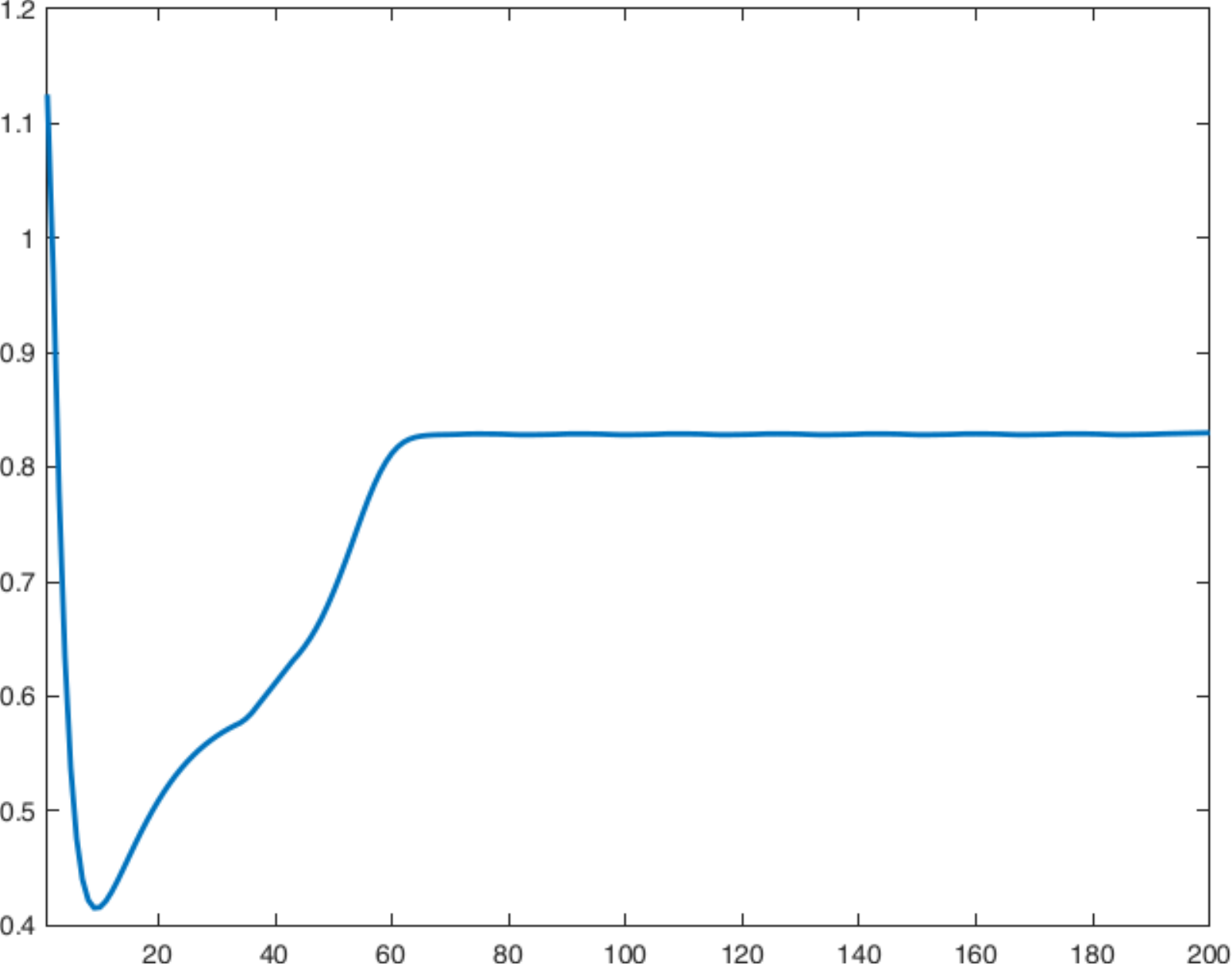} &  
 \includegraphics[height=4cm]{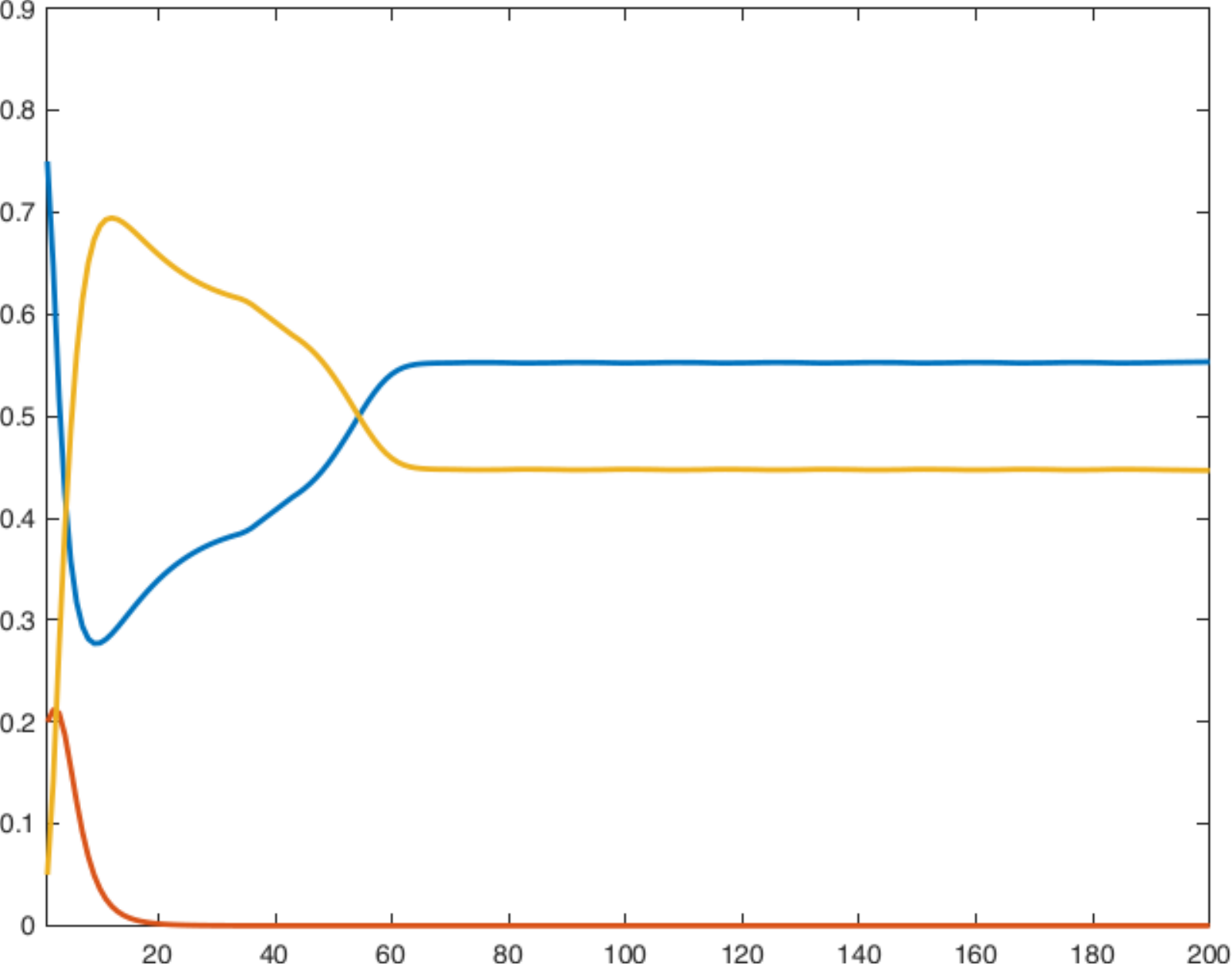}  
\end{tabular}
\caption{Comparison between the optimal social planner vaccination policy when $\mathcal{R}_o=3$ (first row) or $\mathcal{R}_o=3/2$ (second row).
The figures in the first column show the evolution of the vaccination policy through the optimal control $u_t$; the ones in the second column show the evolution of the instantaneous reproduction number $\mathcal{R}_t=\frac{\beta}{\gamma}S(t)$; 
the ones in the third column show the evolution of the percentage of susceptible (in blue), infected (in red) and recovered (in green) individuals.}
\label{fig6}
\end{center}
\end{figure}

It is interesting to notice that the two choices of $\mathcal{R}_o$ are such that the dynamical system of susceptible and infected stabilizes around the equilibria in which the disease becomes endogenous or achieves zero infections. 
As a matter of fact, when $\mathcal{R}_o={3}$, the optimal strategy is constantly equal to $\overline{U}=\frac{7}{120}$ and the dynamical system converges to the equilibrium point $(S_{\infty}^{(1)},I_{\infty}^{(1)})\approx(0.33,0.11)$. 
On the other hand, if $\mathcal{R}_o=\frac{3}{2}$, the optimal strategy stabilizes around to the value $u^\star_{\infty}\approx 0.028$ and the dynamical system converges to the equilibrium point $(S_{\infty}^{(2)},I_{\infty}^{(2)})\approx (0.56, 0)$. Hence, a lower natural reproduction number $\mathcal{R}_o$ has the effect of making it possible to asymptotically keep the number of infected to zero through the optimal vaccination policy.

\subsection{Optimal vaccination policy for different values of initial infected individuals}
\label{differentI0}

In this section we study how the optimal vaccination policy reacts to different initial percentages of the infected population.
All the parameters are set as in Section \ref{OptimalVsNo}, $\frac{\beta}{\gamma}=1.5$, and the percentage $I(0)$ of the initial infected 
individuals is $1\%$, $5\%$ and $10\%$. The results of the numerical study are presented in Figure \ref{fig7}.

We observe that, in all the considered cases, the optimal vaccination rate asymptotically stabilizes around the equilibrium level $u_{\infty} \approx 0.028$, so that, under the optimal vaccination policy, $I_{\infty}=0$ after circa 20 weeks. In all the cases, the vaccination policy starts at the maximal rate, then the vaccination campaign is relaxed, and it is finally kept at the constant rate $u_{\infty} \approx 0.028$. However, the starting date of this final phase of constant vaccination rate is different: it is circa $15$ weeks when $I(0)=1\%$; it is circa $30$ weeks when $I(0)=5\%$; it is circa $54$ weeks if $I(0)=10\%$. That is, the social planner is allowed to stabilize the vaccination rate only after a first time period, whose length increases when the number of initial infected people increases.

\begin{figure}[htb]
\begin{center}
\begin{tabular}{cccc}
\includegraphics[height=4cm]{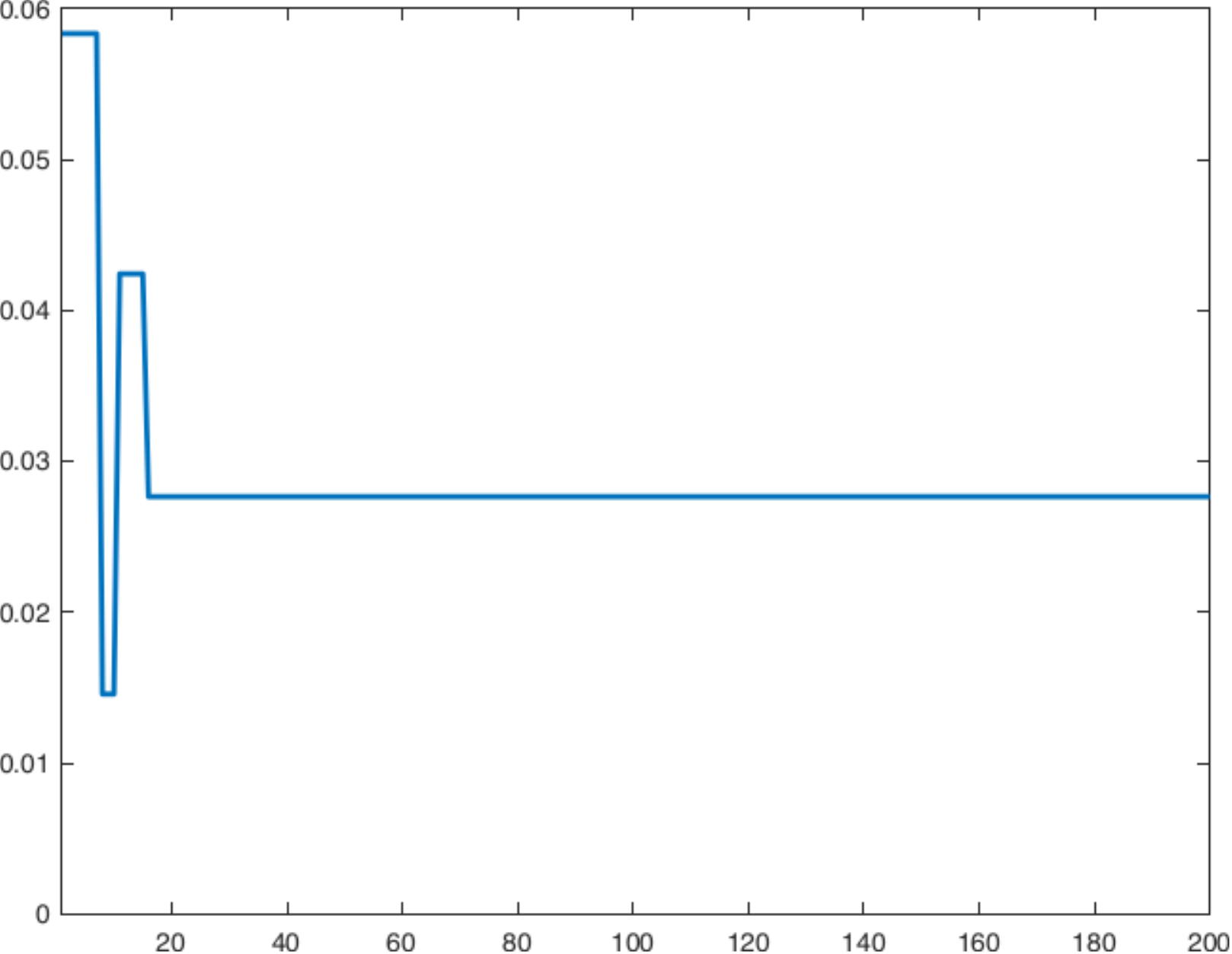} &
 \includegraphics[height=4cm]{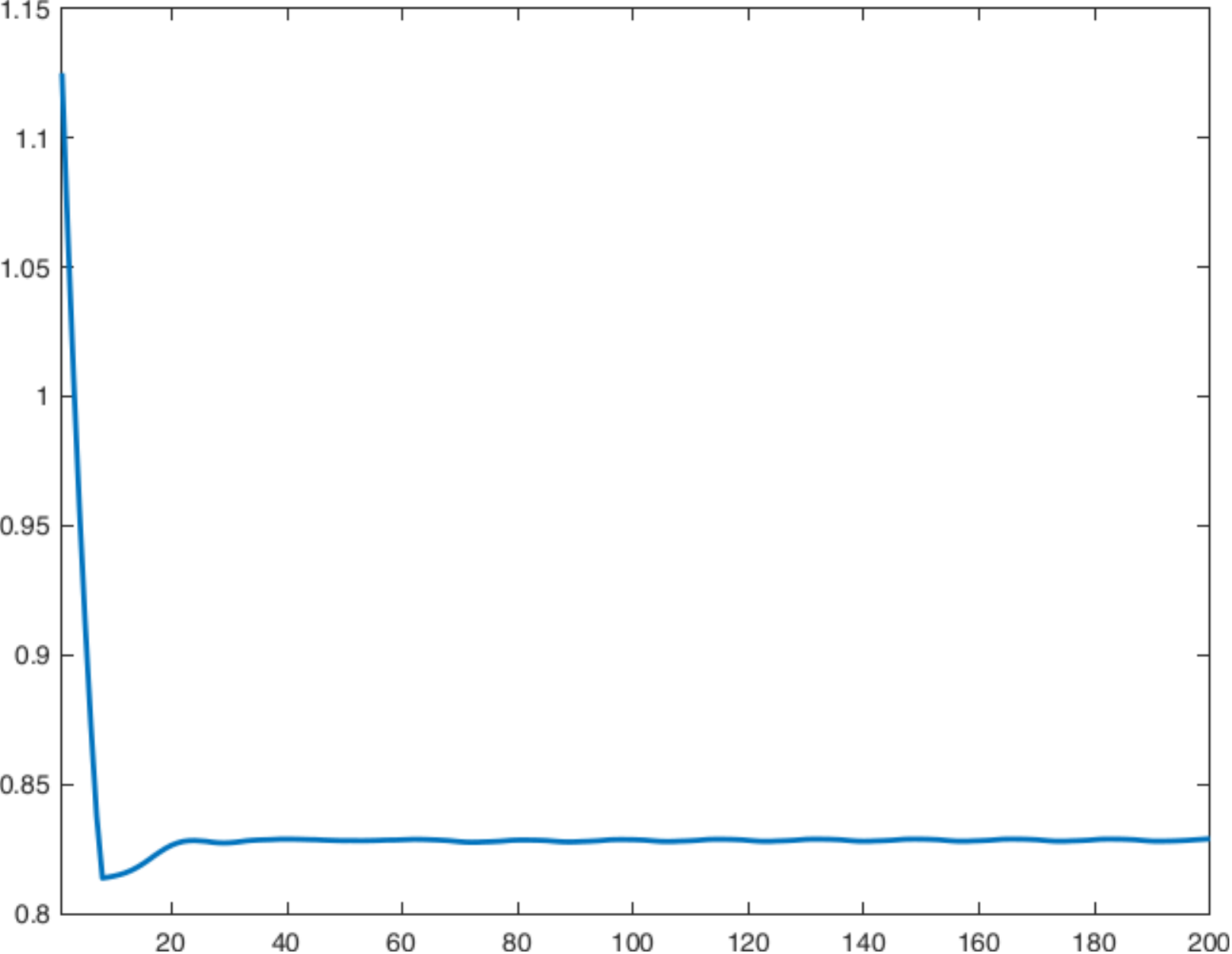} &  
 \includegraphics[height=4cm]{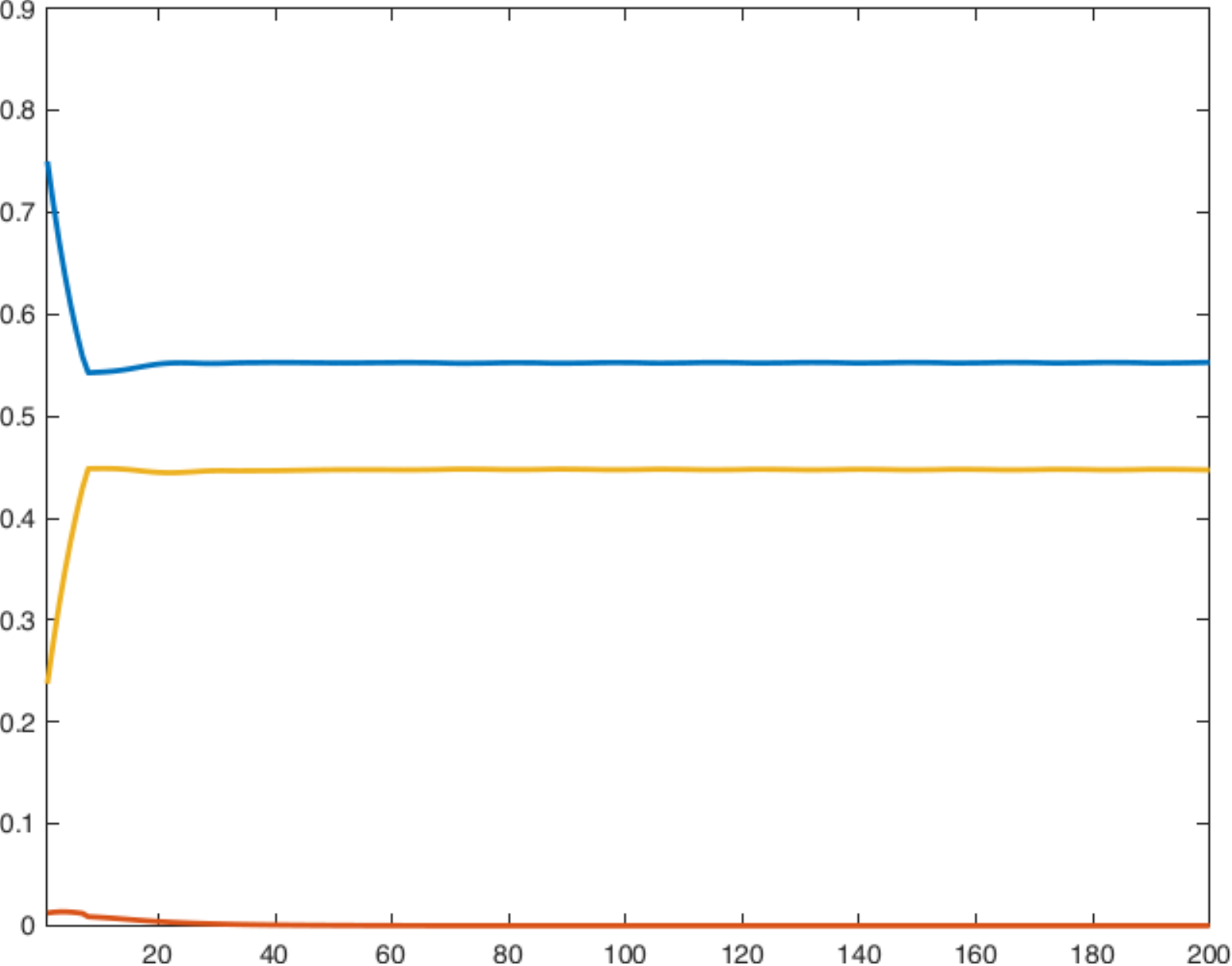}  
\\ \\
\includegraphics[height=4cm]{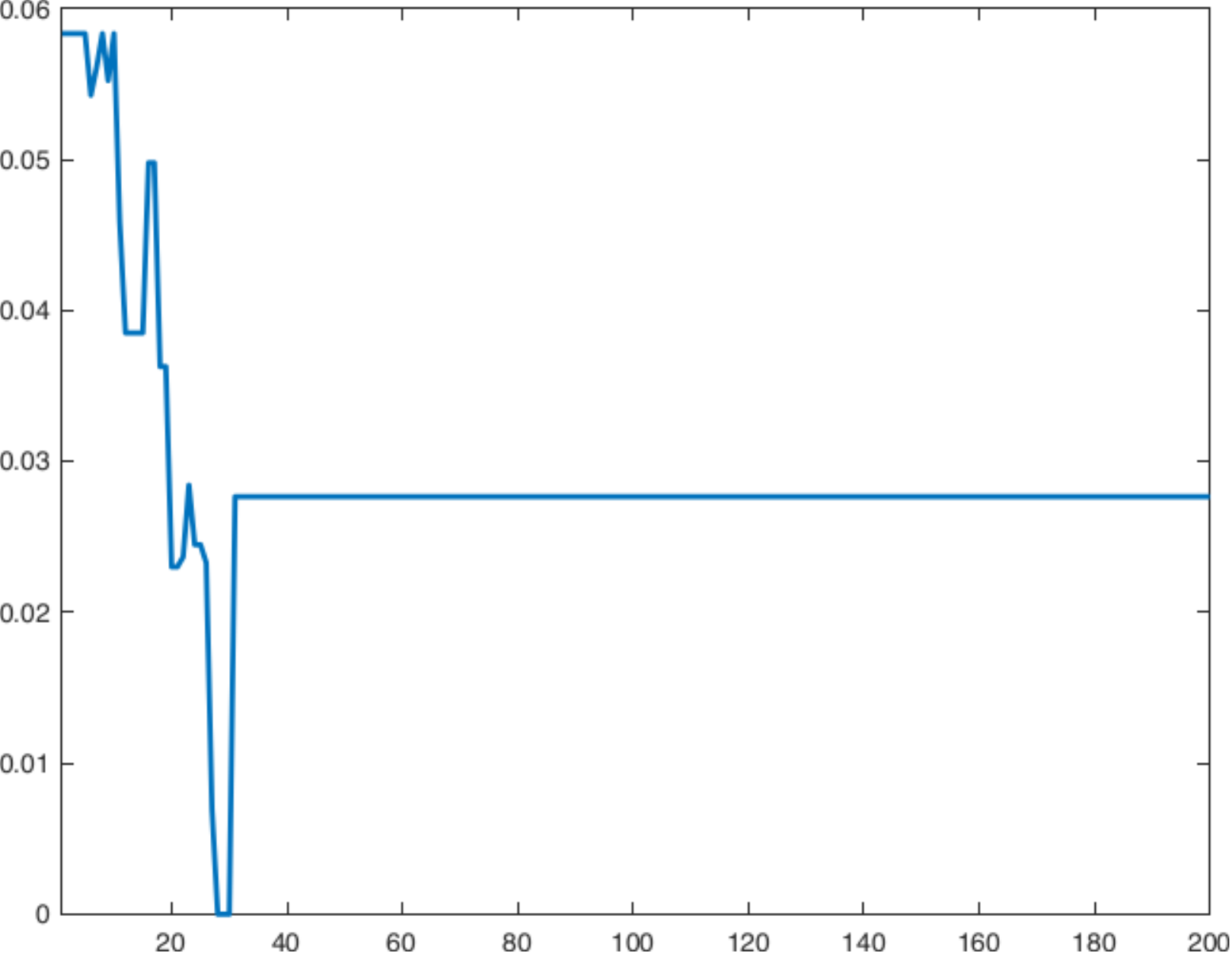} &
 \includegraphics[height=4cm]{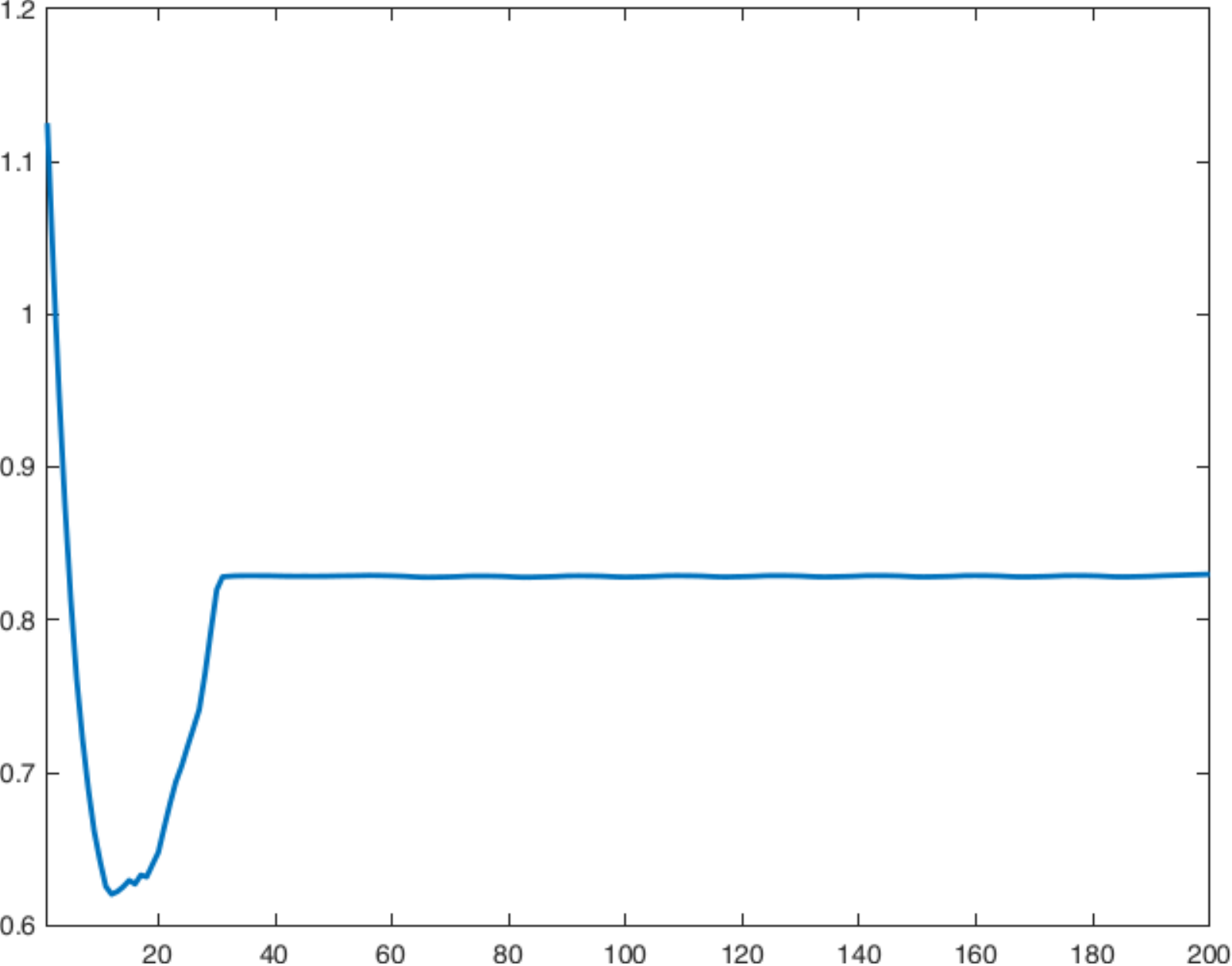} &  
 \includegraphics[height=4cm]{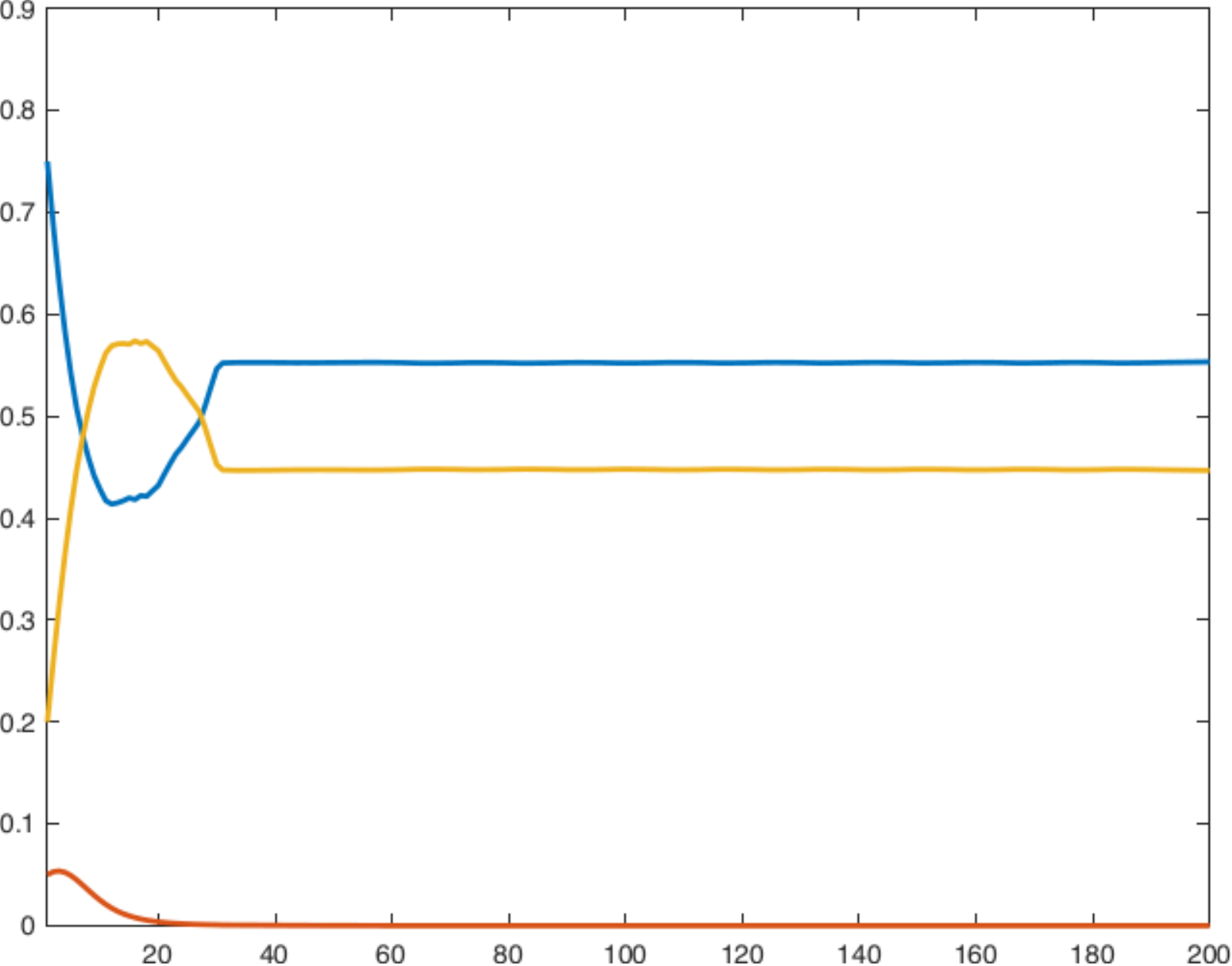}  
\\ \\
\includegraphics[height=4cm]{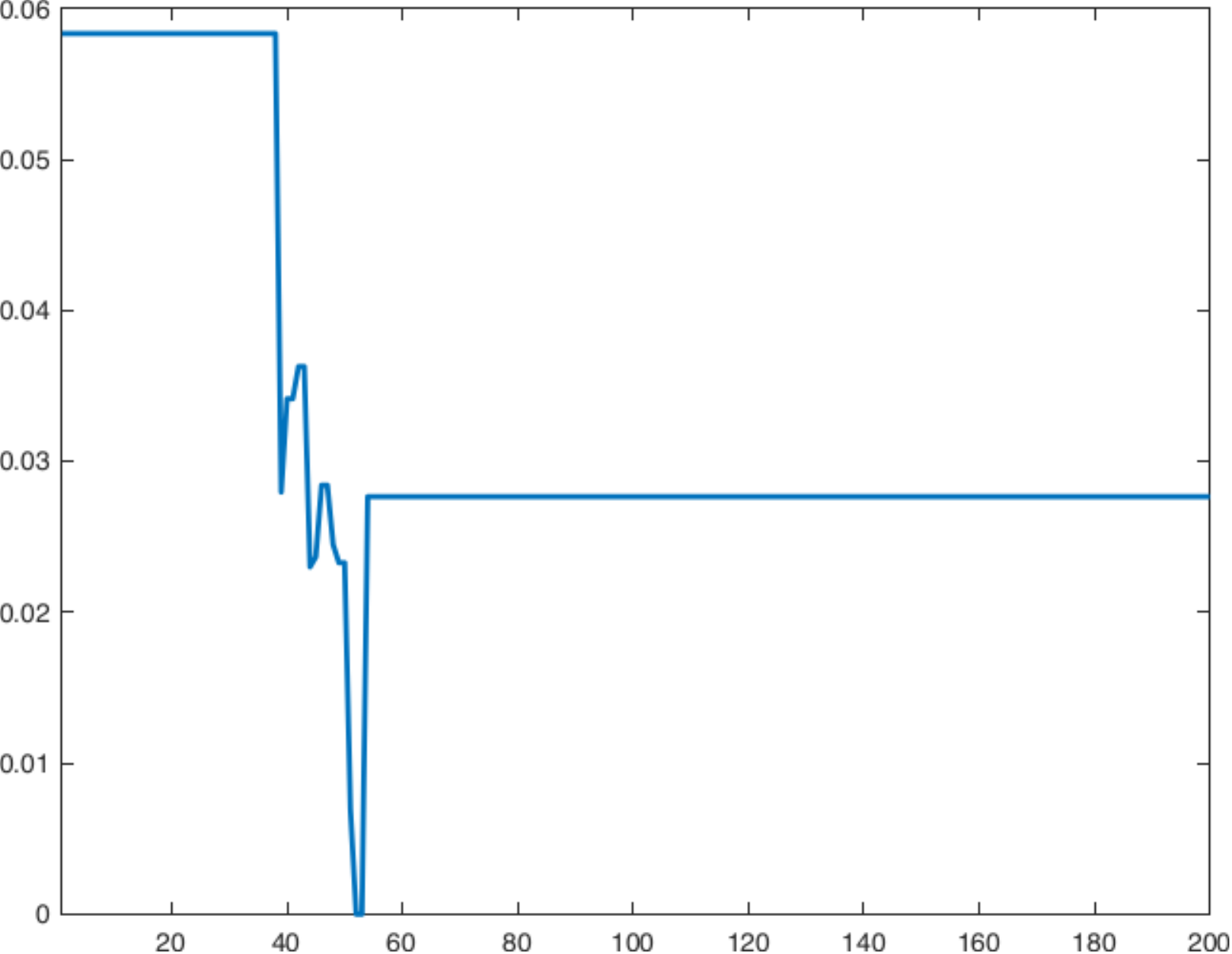} &
 \includegraphics[height=4cm]{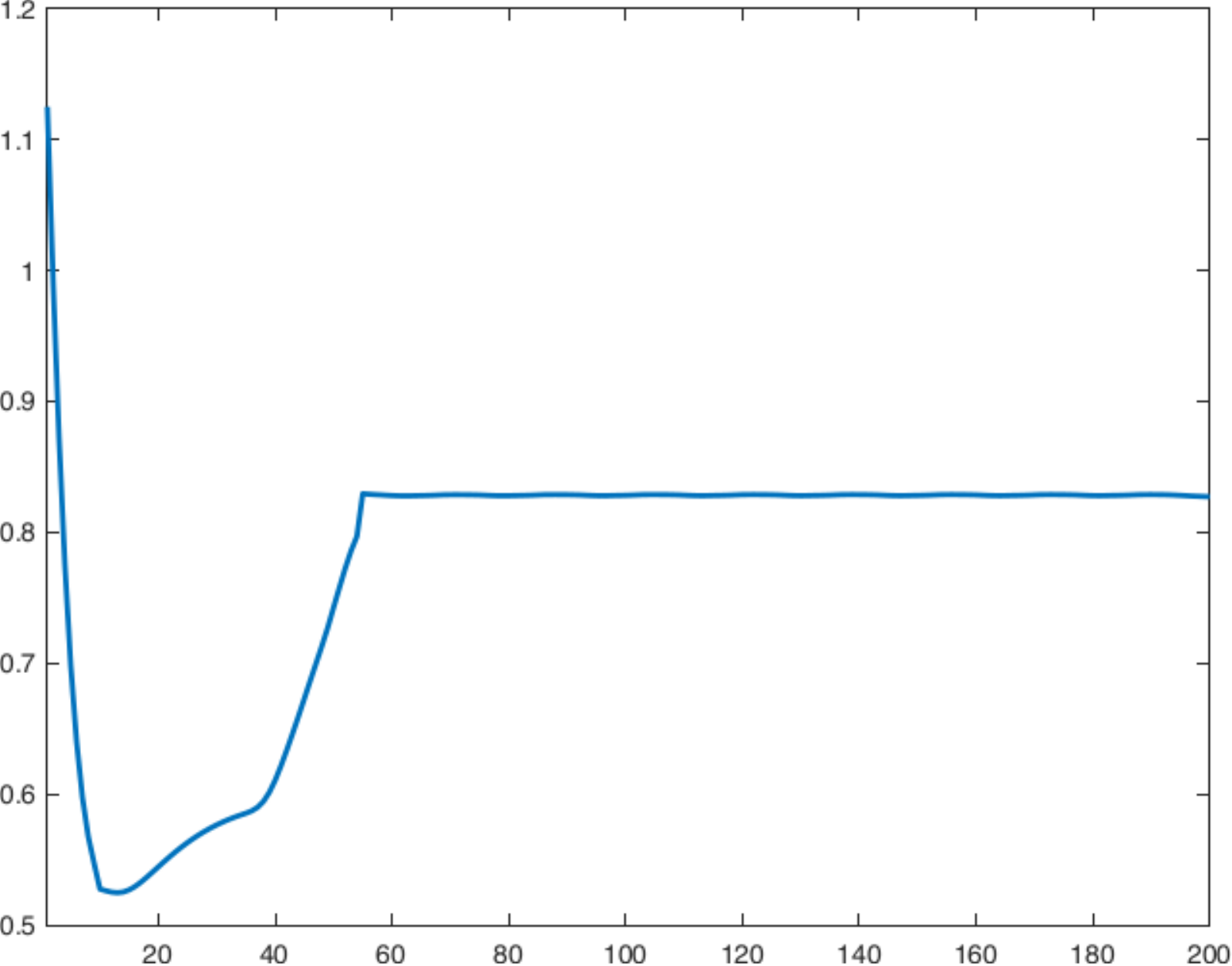} &  
 \includegraphics[height=4cm]{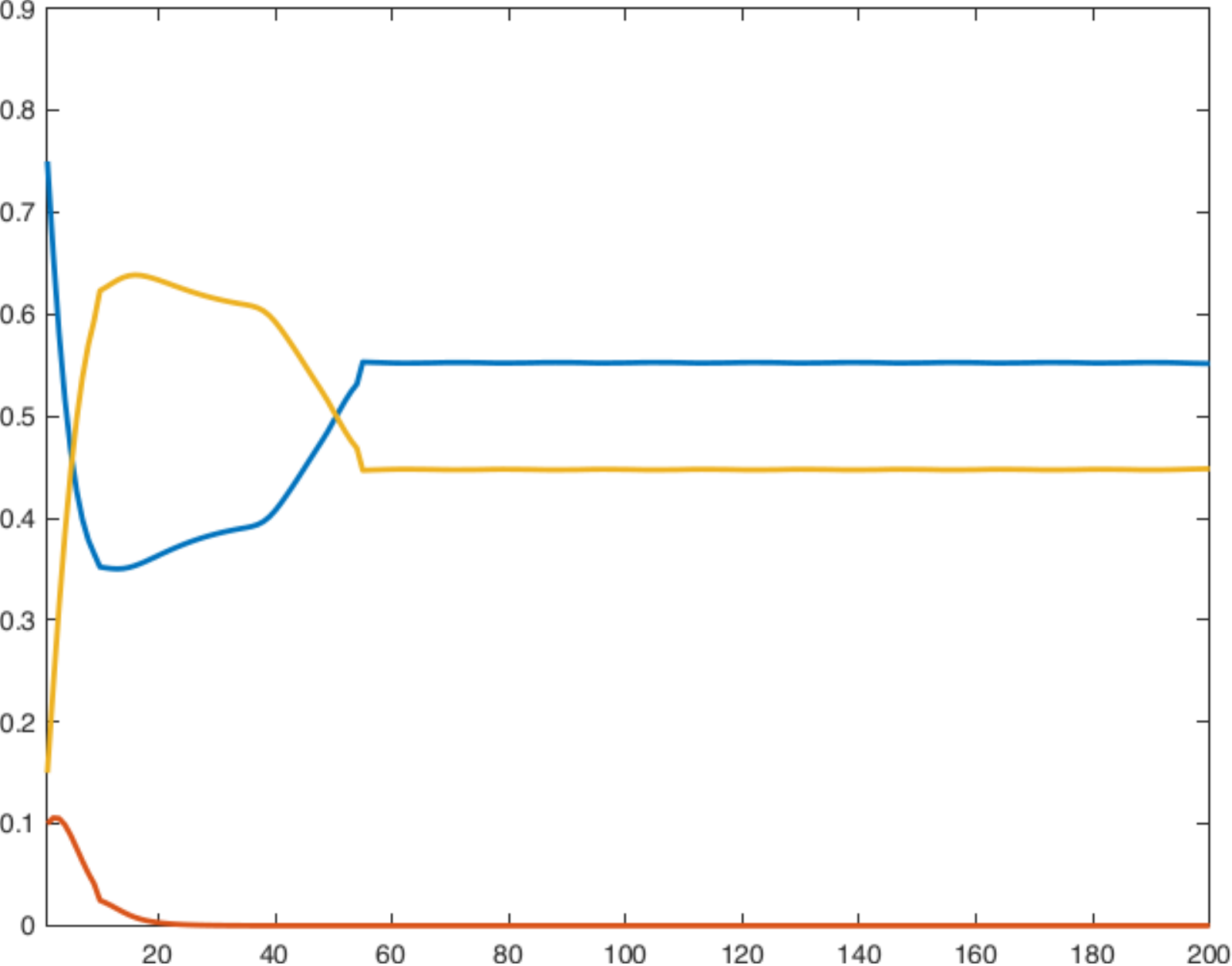}  
\end{tabular}
\caption{Comparison between the optimal social planner vaccination policy in the case the initial infected individuals are $1\%$ (first row), $5\%$ (second row) and
$10\%$ (third row).
The figures in the first column show the evolution of the vaccination policy through the value of the optimal control $u_t$; the ones in the second column show  
the evolution of the instantaneous reproduction number $\mathcal{R}_t=\frac{\beta}{\gamma}S(t)$; 
the ones in the third column show the evolution of the percentage of susceptible (in blue), infected (in red) and recovered (in green) individuals.}
\label{fig7}

\label{fig7}
\end{center}
\end{figure}


\section{Conclusions}
\label{sec:concl}

Within a SIR model with reinfection possibility, we have considered the problem of a social planner that aims at reducing the number of individuals susceptible to an infectious disease via a vaccination policy that minimizes social and economic costs. The resulting optimal control problem is nonlinear and nonconvex, due to the nonlinear dynamics of the percentage of susceptible, infected and recovered (immunized) populations. Combining refined techniques of viscosity theory and results from the theory of Regular Lagrangian flows, we are able to provide verifiable sufficient conditions under which the minimal cost function identifies with the unique semiconcave viscosity solution to the corresponding HJB equation. A recursion on the latter is then employed in order to provide a numerical implementation in a case study with quadratic costs. Our experiments show that it is possible to maintain asymptotically the number of infected to zero, at least when the disease is not too infective. However, given that the model allows reinfection with positive probability, this comes at the cost of keeping vaccinating at a constant rate in the long-run.

There are several directions towards which this research can be extended and continued. First of all, from a technical point of view, it would be interesting to grasp a deeper understanding of how the theory Regular Lagrangian Flows can be helpful for proving the well-posedness of the closed-loop equation arising in nonlinear optimal control problems, as those arising in mathematical epidemiology. Second of all, from a modeling perspective, stochastic and partial observation features should be included in the model, as the evolution of an infectious disease is clearly far to obey completely observable deterministic law of motions. Thirdly, it would be intriguing to study the problem of optimal vaccination problem from a moral-hazard point of view, thus leading to a principal-agent problem where the social planner designs benefits which should induce the agents (the susceptible population) to get vaccinated. These aspects are clearly outside the scope of the present work and are therefore left to future research.


\appendix

\section{A Technical Results}
\label{sec:app}

\begin{lemma}
\label{lemm:SR-nonneg}
Recall \eqref{eq:S}, \eqref{eq:I} and \eqref{eq:R}. One has $S(t)>0$, $I(t)>0$ and $R(t)>0$ for all $t\geq0$.
\end{lemma}
\begin{proof}
Let $(S(0),I(0),R(0))=:(s,i,r) \in (0,1)^3$ and define $\tau:=\inf\{t\geq0:\, S(t) \leq 0\} \wedge \inf\{t\geq0:\, I(t) \leq 0\} \wedge \inf\{t\geq0:\, R(t) \leq 0\}.$
For the sake of contradiction, suppose that $\tau < \infty$. Then, let $t\in[0,\tau)$ and, by using \eqref{eq:S}, \eqref{eq:I} and \eqref{eq:R}, the chain rule yields:
\begin{eqnarray}
& \displaystyle \ln\big(S(t)I(t)R(t)\big) = \ln(s,i,r) + \int_0^t \Big( - \beta I + \eta \frac{R}{S} + \beta S - (\gamma + \eta + u) + \gamma \frac{I}{R} + u \frac{S}{R}\Big)(q) \d q \nonumber \\
& \displaystyle \geq \ln(s,i,r) - \beta \int_0^t I(q) \d q - (\gamma + \eta + \overline{U}) t \geq \ln(s,i,r) - (\gamma + \eta + \overline{U}) t - t \max_{q \in [0,t]}I(q),
\end{eqnarray}
where in the penultimate inequality we have used that $0 \leq u(\cdot) \leq \overline{U}$. The contradiction now follows by taking limits as $t \to \tau$.
\end{proof}

\textbf{Conflict of interest statements.} The authors  certify that they have no affiliations with or involvement in any organization or entity with any financial or non-financial interest in the subject matter or materials discussed in this manuscript. 





\end{document}